\documentclass[preprint,12pt]{elsarticle}
\journal{Bulletin des sciences mathématiques}
\usepackage[unicode=true]{hyperref}
\usepackage{amsfonts}
\usepackage{amsmath}
\usepackage{amssymb}
\usepackage{amsthm}
\usepackage{amstext}
\usepackage{graphics,graphicx,epsfig}
\usepackage{cite}
\usepackage{calc}
\usepackage{hyperref}
\usepackage{latexsym}
\usepackage[margin=1.0in]{geometry}
\usepackage{comment}

\theoremstyle{plain}
\newtheorem{theorem}{Theorem}[section]
\newtheorem{lemma}{Lemma}[section]

\newtheorem*{corollary}{Corollary} 

\theoremstyle{definition}
\newtheorem*{mydef}{Definition}

\theoremstyle{remark}
\newtheorem*{remark}{Remark}
\newtheorem*{note}{Note}

\renewcommand\Re{\operatorname{Re}}
\renewcommand\Im{\operatorname{Im}}

\def\R{\mathbb{R}}
\def\Q{\mathbb{Q}}
\def\Z{\mathbb{Z}}
\def\N{\mathbb{N}}
\def\C{\mathbb{C}}

\def\cA{\mathcal{A}}
\def\cB{\mathcal{B}}
\def\cT{\mathcal{T}}
\def\rT{\mathrm{T}}
\def\ri{\mathrm{i}}
\def\rd{\mathrm{d}}
\def\rg{\mathrm{g}}
\def\bM{\mathbf{M}}
\def\bv{\mathbf{v}}

\def\mod#1{\,({\rm mod\ }#1) }
\def\revd{\overleftarrow{d}\hskip-1pt}
\def\revZ{\overleftarrow{Z}\hskip-1pt}
\def\revGamma{\overleftarrow{\Gamma}\hskip-1pt}

\def\qad{\hskip 5pt}

\begin{document}
\hyphenation{build}

\begin{frontmatter}



\title{B\'ezier curves and the Takagi function}

\author[a]{Lenka Pt\'a\v ckov\'a\corref{cor1}}
\ead[a]{lenka@kam.mff.cuni.cz}
\address[a]{Dept. of Applied Mathematics, Charles University, Ke Karlovu 3, 121 16 Praha 2, Czech Republic}
\cortext[cor1]{Corresponding author (Lenka Pt\'a\v ckov\'a)}

\author[b]{Franco Vivaldi}
\ead[b]{f.vivaldi@maths.qmul.ac.uk}

\address[b]{School of Mathematical Sciences, Queen Mary, University of London, London E1 4NS, UK}

\begin{abstract}
We consider B\'ezier curves with complex parameters, and we determine
explicitly the affine iterated function system (IFS) corresponding to 
the de Casteljau subdivision algorithm, together with the complex 
parametric domain over which such an IFS has a unique global 
connected attractor. For a specific family of complex parameters having
vanishing imaginary part, we prove that the Takagi fractal curve
is the attractor, under suitable scaling.
\end{abstract}

%


\end{frontmatter}

%
%

\section{Introduction}
We consider connections between a subdivision scheme 
of geometric modelling, the de Casteljau algorithm for 
B\'ezier curves, and iterated function systems of fractal theory.
We obtain an explicit representation of this algorithm as a pair of affine 
maps with coefficients in $\Z[t]$, the ring of polynomials in $t$ with 
integer coefficients. Here $t$ is the parameter of the B\'ezier curve,
and fractals are generated by allowing $t$ to assume complex values.
After determining the range of $t$-values for which the IFS has a unique
global connected attractor, we establish rigorously the appearance of 
the Takagi fractal curve in a specific asymptotic regime.

The Takagi curve is the graph of a continuous and nowhere differentiable 
function, the Takagi function, introduced by Teiji Takagi in 1903, and 
subsequently rediscovered several times, e.g., by Hildebrandt in 1933 
or de Rham in 1957 \citep{Lagarias}. 
The Takagi function appears in many areas of mathematics, including analysis, probability 
theory, combinatorics, and number theory ---see the surveys 
\citep{Allaart,Lagarias} and references therein.

In geometric modelling, the recursive construction of fractals has
proved valuable for generating realistic (i.e., typically
non-smooth) forms. This approach originated from the work of Barnsley 
and co-workers on fractal image compression \citep{BarnsleyHurd}, 
using the machinery of iterated function systems; these are dynamical 
systems consisting of collections of contraction mappings 
\citep[section 3.7]{FractalsEverywhere}.
This idea later appeared in the context of subdivision schemes,
such as B\'ezier and B-spline curves, which are widely used in 
geometric modelling, computer graphics, and
computer aided design, due to the simplicity of their construction.

The first such application is due to R.~Goldman \citep{Goldman2004}, 
who presents a constructive procedure that builds an IFS based on the 
de Casteljau subdivision algorithm, and shows that B\'ezier curves are 
attractors of such an IFS.

Schaefer et al.~\citep{SchaeferGoldman2005} extend the idea by showing that 
curves and surfaces generated by several well-known subdivision algorithms 
are also attractors, fixed points of IFSs. 
They demonstrate how any curve generated by an arbitrary stationary subdivision scheme or subdivision surfaces without extraordinary\footnote{An extraordinary vertex is a vertex whose valence $\neq 6$ for triangle meshes or valence $\neq 4$ for quad meshes. Vertex valence is the number of vertices directly adjacent to a vertex.} vertices can be represented by an IFS.
Furthermore, they provide a general paradigm for introducing control points and subdivision rules for arbitrary fractals generated by an IFS consisting of affine transformations, such as the Sierpinski gasket or the Koch curve. Moving the control points of these fractals induces transformations that are affine (but not necessarily conformal).

Tsianos and Goldman \citep{Tsianos2011} study the opposite direction: 
they show how to apply the de Casteljau subdivision algorithm and several standard knot insertion procedures for B-splines to build fractal shapes. 
They do so by allowing the parameters, or knots, to be complex numbers. 
Then they construct IFSs for these extended versions of classical subdivision 
algorithms, where the control points and the parameters are taken from 
the complex plane. They prove that starting with any control polygon, 
the IFS for B\'ezier curves with a given complex parameter converges 
uniformly to a limiting curve.
Furthermore, they show that every fractal in the plane generated by a 
conformal IFS can be reproduced by B\'ezier subdivision in the complex 
plane \citep[Corollary 3.3]{Tsianos2011}. 
They can thus generate fractals with control points; however, 
in contrast to \citep{SchaeferGoldman2005}, adjusting the control points 
induces transformations that are conformal but not necessarily affine. 

Some well-known subdivision schemes can also produce fractal curves and 
surfaces. For example, the four-point interpolating subdivision scheme 
introduced by Dyn et al. \citep{Dyn} is defined using a tension parameter 
$\omega$, and the smoothness of the limit curves depends on this parameter. 
Thus while the limit curve is almost $C^2$-continuous for some
values of $\omega$, for other values one obtains
fractal curves and surfaces, as shown in \citep{fractalSubd}.

Connections between splines and fractals also emerge if one allows the order 
of a polynomial B-spline to be a real or complex number, leading to 
fractional and complex B-splines, respectively ---see \citep{FBU06,Massopust2019} and references therein.
The complexification of the order is also available for the
\textit{exponential} B-splines, defined as convolution products of 
exponential functions, see \citep{Massopust2014, Massopust2019} and 
references therein. 
However, according to \citep{Massopust2019}, neither the classical nor 
extended polynomial and exponential B-splines provide appropriate 
approximations of functions that exhibit self--similar or fractal 
behaviour. In these cases, one needs to resort to fractal interpolation 
and approximation techniques. 
The extension of polynomial B-splines to self--similar or fractal 
functions was presented in \citep{Massopust} (see also
\citep[Section 5]{Massopust2019}).

The structure and main results of this paper are as follows. 
In section \ref{sec:IFS} we briefly review the theory of iterated function 
systems. Alongside we provide some constructs and results on metric spaces.
Our focus is on fractals generated by IFS consisting of affine 
transformations.

In section \ref{sec:IFSforSubdivision}
we provide some background on B\'ezier curves, adopting
the definition of the de Casteljau subdivision 
for curves with complex parameter $t$ and $n$ control points 
given in \citep{SchaeferGoldman2005}. 
We then obtain an explicit representation of the de Casteljau 
scheme as a pair of affine maps on $\Z[t]^n$, and we determine the complex 
$t$-domain over which such an IFS has a unique connected attractor 
(theorem \ref{thm:IFSAttractor}).
This result rests on the key lemma \ref{lma:BezierUpperTriangular},
which provides a novel upper-triangular representation of the de 
Casteljau matrices.

In section \ref{sec:Takagi} we consider the de Casteljau IFS, 
with two control points. For parameters of the type $t=1/2+\ri\beta$ we
associate to each binary code a curve ---parametrised by $\beta$---
that describes the location of the point with that symbolic address 
on the attractor. The derivative of these curves at $\beta=0$ define 
a vector field on the smooth B\'ezier curve, and we show that the
Takagi function gives the amplitude of this field (lemma \ref{lma:v}).
Using this result we then prove that a suitably scaled version of the 
attractor of the IFS in the $\beta\to 0$ limit is the Takagi curve 
(theorem \ref{thm:Takagi}).
In the case of an arbitrary number of control points, the vector field
has several components, and we show that one component is still
given by the Takagi function (with a different scaling).

\section{Iterated function systems}\label{sec:IFS}
We provide some background on iterated function systems 
 ---see \citep{FractalsEverywhere}. For basic dynamical systems
terminology, see also \citep{KatokHasselblatt}.

An iterated function system is a dynamical system whose 
`points' are subsets of a set, $\C^m$ in the present setting. 
The appropriate space for such a system is a \textit{complete metric space},
namely a pair $(X,d)$, where $X$ is a set, and $d$ is a metric on $X$ with respect to which all Cauchy sequences 
in $X$ converge to a limit in $X$.

Further, a mapping $f$ of $(X,d)$ to itself is a
\textit{contraction mapping} if there is a non-negative real 
number $s <1$ such that for all $x,y\in X$
\begin{equation*}
    d(f(x),f(y))\leqslant s \cdot d(x,y) \qquad 0\leqslant s<1.
\end{equation*}
The smallest such number $s$ is called the \textit{contractivity factor} 
for $f$. 

The relevance of contraction mappings in our context is given by
the \textit{Banach fixed point theorem} \citep[Section 5.1]{Kreyszig}.

\begin{theorem}\label{thm:BanachFixedPoint}
A contraction mapping of a complete metric space has a unique fixed point.
Such a fixed point is an attractor, whose basin of attraction is the whole 
space.
\end{theorem}

Fractals are compact sets, and we are interested in the set $\mathcal{H}(X)$
of all compact subsets of a set $X$. If $(X,d)$ is a complete metric space,
one defines a metric $d_H$ on $\mathcal{H}(X)$ ---the \textit{Haussdorf 
metric}--- which turns it into a complete metric space. 

This is done as follows.
Given $x \in X$ and $B \in \mathcal{H}(X)$, we first define the distance
from $x$ to $B$ as $d(x,B)=\min\{d(x,y)\,|\, y \in B \}$.
Then, given $A\in \mathcal{H}(X)$ we define the distance from $A$ to $B$ as
\begin{equation}\label{eq:dist2}
    d(A,B)=\max\{d(x,B)\,|\, x \in A \}.
\end{equation}
Finally, by making the above symmetric, one obtains the Hausdorff metric 
$d_H$ on $\mathcal{H}(X)$:
\begin{equation}\label{eq:HausDistance}
    \rd_H(A,B)=\max \{ d(A,B),d(B,A) \}.
\end{equation}
The metric space $(\mathcal{H}(X),\rd_H)$ is complete 
\citep[Section 7]{FractalsEverywhere}, and we
now define a contraction mapping on $\mathcal{H}(X)$.

\begin{mydef}
Let $(X,d)$ be a complete metric space and let 
$f_k:X\rightarrow X,\; k=1,\dots N$ be a collection of 
contraction mappings with respective contractivity factors $s_k$. 
The dynamical system
\begin{equation}\label{eq:IFS}
F: \mathcal{H}(X) \rightarrow \mathcal{H}(X)\hskip 30pt
F(A)=\bigcup_{k=1}^N f_k(A)
\end{equation}
is called a (hyperbolic) \textit{iterated function system (IFS)} on $X$.
\end{mydef}

Then we have \citep{Hutchinson}:
\begin{theorem}[Hutchinson]
The transformation $F$ given in (\ref{eq:IFS}) is a contraction mapping 
on the complete metric space $(\mathcal{H}(X),\rd_H)$ with contractivity 
factor $s=\max\{s_k\;|\; k=1,\dots,N\}$.
\end{theorem}
From the above and Banach theorem \ref{thm:BanachFixedPoint}, 
we conclude that $F$ has its unique fixed point $A^{\ast}\in \mathcal{H}(X)$, 
which obeys
\begin{displaymath}
A^{\ast}=F(A^{\ast})=\bigcup_{i=1}^N f_k(A^{\ast}),
\end{displaymath}
and can be obtained as the limit
\begin{equation*}
   A^{\ast}=\lim_{n\rightarrow\infty} F^n(B)\qad\text{for any}\qad B\in \mathcal{H}(X).
\end{equation*}
The above fixed point will be called the \textit{attractor of the IFS}.

Many well-known fractals are generated by affine 
iterated function systems. We consider maps of the form
\begin{equation}\label{eq:AffineTrans}
f: \C^m \rightarrow \C^m\hskip 30pt
    f(\mathbf{x})=\mathbf{A}\mathbf{x} + \mathbf{b},
\end{equation}
where $\mathbf{A}$ is an $m\times m$ regular matrix and $\mathbf{b}$ 
is an $m$-dimensional vector.


\section{IFS for subdivision curves}\label{sec:IFSforSubdivision}

Subdivision curves are defined recursively. The process starts with a 
set of control points $\mathbf{P}$ as input, a refinement scheme is applied 
to them, generating as output a new refined set of control points. The subdivision curve is the limit of this recursive
process.

In \citep{SchaeferGoldman2005, Tsianos2011} a procedure is presented 
to generate an IFS for binary subdivision curves, including uniform 
B-spline curves and B\'ezier curves.
We provide an alternative representation of the IFS of B\'ezier curves,
and prove that in a suitable range of complex parameters this IFS is 
contractive and therefore has a unique fixed point, the 
global attractor. We also establish that in the same parametric 
domain the attractor is connected.

We first review some basic facts concerning B\'ezier curves.

\subsection{B\'ezier Curves and the de Casteljau subdivision algorithm}\label{sub:Casteljau}
B\'ezier curves were popularized in 1962 by the French engineer Pierre B\'ezier, who worked for Renault. They can be expressed explicitly as parametric curves using the Bernstein basis polynomials.
Alternatively, points on a B\'ezier curve can be constructed by the de Casteljau algorithm, which was introduced by Paul de Casteljau already in 1959. He worked for Citro\"en and his work was kept a secret by Citro\"en for a long time \citep{HandbookCAGD}. P. de Casteljau and P. B\'ezier had different approaches and they introduced B\'ezier curves independently, but P. B\'ezier could publish his work first.

The de Casteljau algorithm generalises to polynomial curves of arbitrary degree 
the construction of a parabola by repeated linear interpolation \citep{farin}. 
It is defined as follows:
Let $\textbf{p}_0, \dots , \textbf{p}_n \in \mathbb{R}^3$ be a 
set of control points and $t \in [0,1]$. We set
\begin{equation}\label{eq:castel}\left\{ \begin{array}{rcl}
\mathbf{p}_i^0(t)&=&\mathbf{p}_i\\
\mathbf{p}_i^k(t )&=&(1-t )\mathbf{p}_i^{k-1}(t) + t  \mathbf{p}_{i+1}^{k-1}(t)\quad
\left\{ \begin{array}{rcl}
k&=&1,\dots , n\\
i&=&0,\dots , n-k.
\end{array}\right.\end{array}\right.
\end{equation}
Then $\mathbf{p}_0^n(t)$ is the point with parameter value $t$ on the B\'ezier curve of degree $n$.\\

The de Casteljau algorithm can be viewed as a subdivision scheme, a method for finding new control points $\textbf{l}_0(t), \dots , \textbf{l}_n(t)$ and $\textbf{r}_0(t), \dots , \textbf{r}_n(t)$ from the original control points $\textbf{p}_0, \dots , \textbf{p}_n$. 
These new control points, which represent the original B\'ezier curve 
restricted to the parameter intervals $[0,t]$ and $[t,1]$, respectively,
are computed from the formulae
\begin{equation}\label{eq:CastSubdivision}
\begin{array}{ll}
\textbf{l}_i(t)={\displaystyle \sum_{j=0}^i B_j^i(t) \textbf{p}_j },& i=0,\dots,n,\\
\textbf{r}_i(t)={\displaystyle \sum_{j=i}^n B_{n-j}^{n-i}(t) \textbf{p}_{i+n-j} },& i=0,\dots,n,
\end{array}
\end{equation}
where $B_j^n(t)$ are Bernstein basis polynomials $B^n_j(t) = \binom{n}{j}t^j (1-t)^{n-j}$.

The de Casteljau algorithm (\ref{eq:CastSubdivision}) can be rewritten 
into the following matrix form:
\begin{equation}\label{eq:CastelSubd}\begin{array}{rcl}
\left( \begin{array}{c}
\textbf{l}_0\\
\vdots\\
\textbf{l}_n\\
\end{array} \right)
&=&
\left( \begin{array}{cccc}
B^0_0(t) & 0 & \ldots & 0\\
B^1_0(t) & B_1^1(t) & \ldots & 0\\
\vdots & \vdots & \vdots & \vdots \\
B^n_0(t) & B^n_1(t) & \ldots & B^n_n(t)
\end{array} \right)
\left( \begin{array}{c}
\textbf{p}_0\\
\vdots\\
\textbf{p}_n\\
\end{array} \right)
=\textbf{L}(t)\cdot \textbf{P}
\\{} &{}&  {}\\
\left( \begin{array}{c}
\mathbf{r}_0\\
\vdots\\
\mathbf{r}_n\\
\end{array} \right)
&=&
\left( \begin{array}{cccc}
B^n_0(t) & B^n_1(t) & \ldots & B^n_n(t)\\
0 & B_0^{n-1}(t) & \ldots & B_{n-1}^{n-1}(t) \\
\vdots & \vdots & \vdots & \vdots \\
0 & 0 & \ldots & B_0^0(t)
\end{array} \right)
\left( \begin{array}{c}
\mathbf{p}_0\\
\vdots\\
\mathbf{p}_n\\
\end{array} \right)
=\textbf{R}(t)\cdot \textbf{P} \end{array}
\end{equation}

The matrices $\textbf{L}(t)$, $\textbf{R}(t)$ represent the left and right 
subdivision scheme for B\'ezier curves. Starting with the original control 
points and applying these matrices repeatedly generates a sequence of 
control polygons that converge to the original B\'ezier curve.

The matrices $\textbf{L}(t)$ and $\textbf{R}(t)$ are invertible, 
with eigenvalues $\lambda_i=t^i$ and $\lambda_i=(1-t)^i$, respectively,
for $i=0,\dots,n$. 
The eigenvector corresponding to $\lambda_0=1$ is 
$\mathbf{v}_0=(1,1,\dots,1)^\top$, since both matrices are row-stochastic.

\subsection{IFS for B\'ezier curves with complex parameter}
\label{sub:IFSforComplexBezier}
The de Casteljau algorithm (section \ref{sub:Casteljau}) 
provides a subdivision scheme for B\'ezier curves with real parameter $t$. 
In order to generate fractals, we now extend the algorithm to the case of complex 
parameter $t \in \mathbb{C}$, and then transform the matrices 
$\mathbf{L}(t)$ and $\mathbf{R}(t)$ in $\C^{(n+1)\times (n+1)}$ into an IFS 
consisting of two affine maps $f_0$ and $f_1$ of $\C^{(n+1)}$.

This extension is achieved via a conjugacy that displays the $n$-dimensional
linear part $\mathbf{A}_k$ and translation vector $\mathbf{b}_k$ of 
$f_k$ as sub-matrices ---cf.~equation (\ref{eq:AffineTrans}).
This may be done in four equivalent ways:
\begin{equation}\label{eq:affineLR}
\begin{array}{cccc}
\mathrm{I}&\quad\left( \begin{array}{cc}
 \mathbf{A}_k & \mathbf{b}_k \\
 \mathbf{0} & 1
\end{array} \right)\hskip 60pt
&
\mathrm{II}&\quad\left( \begin{array}{cc}
 \mathbf{A}_k & \mathbf{0} \\
  \mathbf{b}_k & 1
\end{array} \right)
\\\noalign{\vskip 10pt}
\mathrm{III}&\quad\left( \begin{array}{cc}
 1 & \mathbf{b}_k \\
 \mathbf{0} & \mathbf{A}_k
\end{array} \right)\hskip 60pt
&
\mathrm{IV}&\quad\left( \begin{array}{cc}
  1 & \mathbf{0} \\
  \mathbf{b}_k & \mathbf{A}_k
\end{array} \right)
\end{array}
\qquad k=0,1.
\end{equation}
These matrices, which we shall call $\mathbf{M}^{(k)}$, $k=0,1$, 
are distinguished by the location of a row, or column,
of the identity matrix $\mathbf{I}_{n+1}$.
These matrices act on the right (I and IV) or on the left
(II and III), on column or row vectors that feature a 1 
---the homogeneous component---
in the last (I and II) or the first entry (III and IV).

The equivalence of these matrices is achieved by a combination of
a reflection $(i,j) \leftrightarrow (n-i,n-j)$ and/or a transposition. 
The former is performed by the matrix with ones on the
secondary diagonal (which coincides with its own inverse), while 
Theorem 1 of \citep{TausskyZassenhaus} guarantees the existence of 
a (symmetric) transposition similarity.

\medskip
A general procedure for constructing a similarity between
the de Casteljau matrices (\ref{eq:CastelSubd}), or more general
subdivision matrices, and the matrices
$\mathbf{M}^{(k)}$ of type II is developed in \citep{SchaeferGoldman2005}:
\begin{equation}\label{eq:BezierIFS}
  \mathbf{M}^{(0)}= \mathbf{P}^{-1}\;\mathbf{L}(t)\mathbf{P},
      \qquad
  \mathbf{M}^{(1)}=\mathbf{P}^{-1}\;\mathbf{R}(t)\mathbf{P},
\end{equation}
where $\mathbf{P}$ is a square matrix created from any column vector 
whose last two entries differ (for instance, the vector 
of control points $(\mathbf{p}_0,\dots,\mathbf{p}_n)^\top\in\C^{(n+1)}$)
by adding columns from identity matrix and a last column of ones 
corresponding to the homogenous component of the coordinates. 
The matrix $\mathbf{P}$, which performs a change of basis in
$\mathbb{C}^{(n+1)\times (n+1)}$, is invertible by construction.
We illustrate the effect of different choices of $\mathbf{P}$
on the attractor of IFS in Figure \ref{fig:cubicBezier}.

\begin{figure}[ht]
\begin{center}
\includegraphics[width=0.32\textwidth]{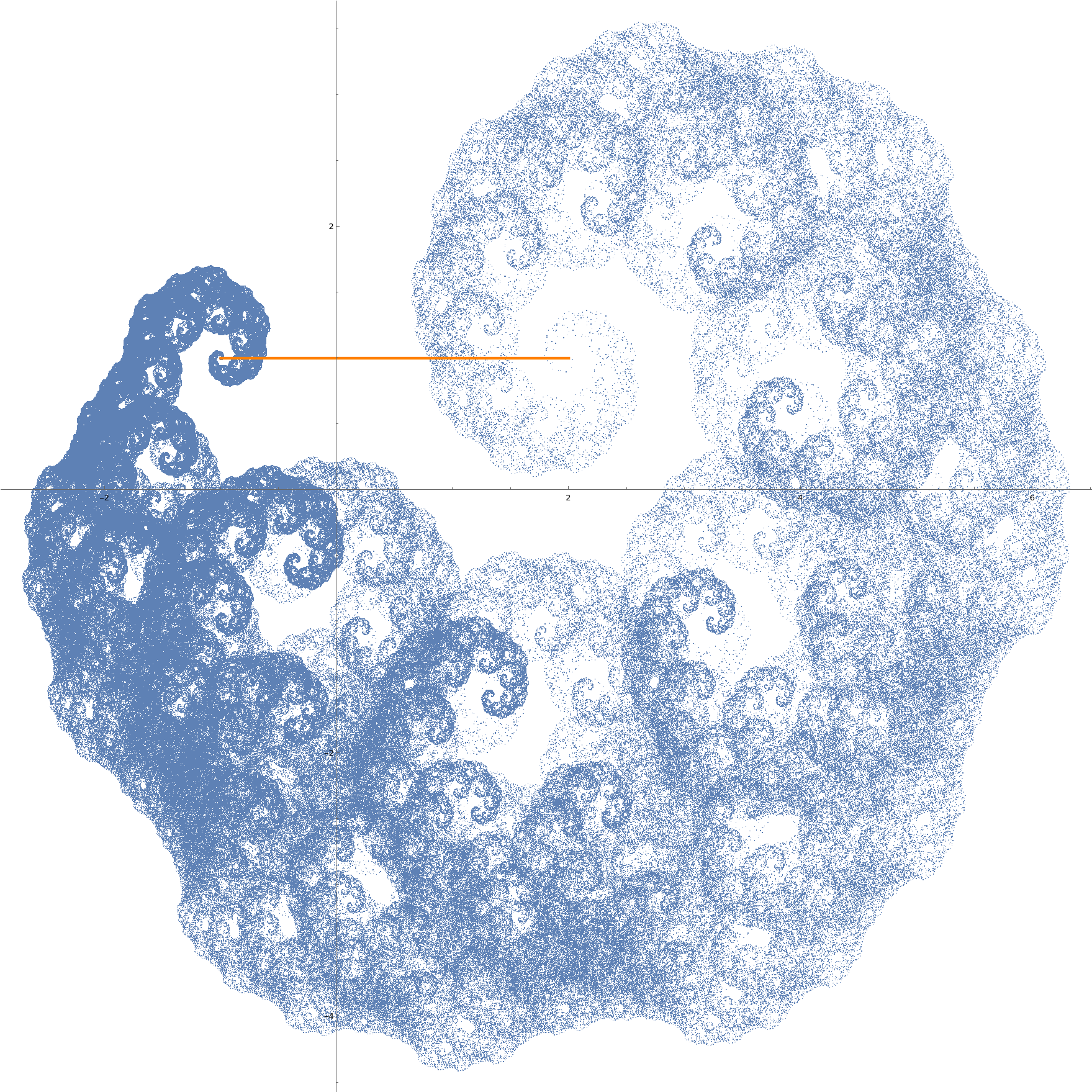}
\includegraphics[width=0.32\textwidth]{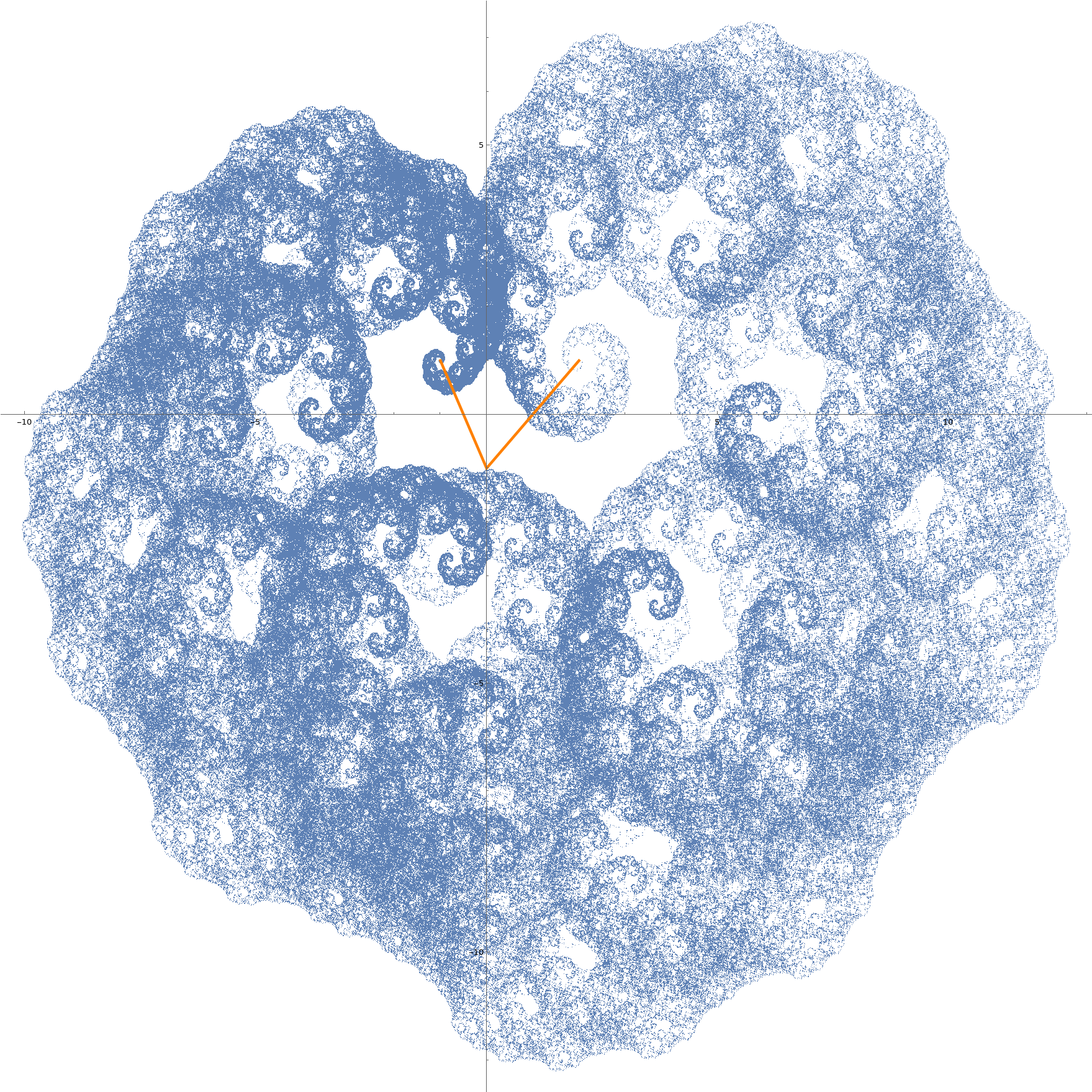}
\includegraphics[width=0.32\textwidth]{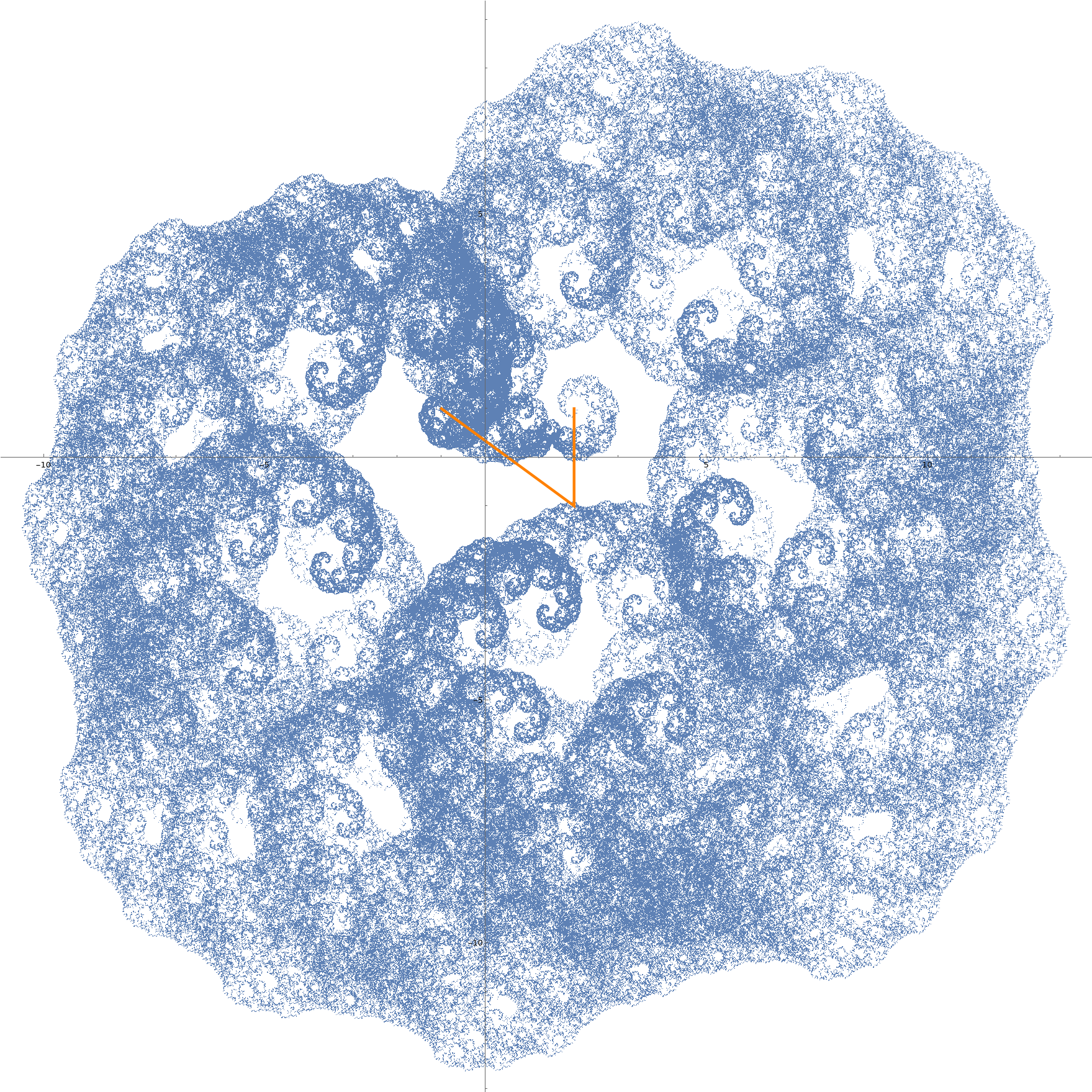}
\caption{Attractors of the IFS for quadratic B\'ezier curve with complex parameter $t = 0.4 -0.55 i$ with control points (plotted as orange polyline) from left to right: 
$(-1 +i, i, 2+i)^\top$, $(-1 +i, -i, 2+i)^\top$, and $(-1 +i, 2-i, 2+i )^\top$. 
The initial point $X\in\C^{1\times (n+1)}$ is an arbitrary point in dimension $n$ in homogenous coordinates. We show the state of the system after 20 iterations. In order to visualize the higher dimensional attractor, we project it to the complex plane. In all cases the fixed point of $f_0$ is the first row of matrix $\mathbf{P}$ of (\ref{eq:BezierIFS}), i.e., $(-1 +i,1,1)$, while the fixed point of $f_1$ is $(2+i,0,1)$ --- the last row of $\mathbf{P}$.}\label{fig:cubicBezier}
\end{center}
\end{figure}

Because the elements in a row of $\mathbf{L}(t)$ (and similarly for 
$\mathbf{R}(t)$) sum up to 1, and the last column of $\mathbf{P}$ is a 
column of ones, the last column of $\mathbf{L}(t)\mathbf{P}$ is also a 
column of ones. 
Further, the last column of $\mathbf{P}^{-1}\mathbf{L}(t)\mathbf{P}$ is the
last column of the identity matrix $\mathbf{I}_{n+1}$ because the last 
column of $\mathbf{P}$ is a column of ones and 
$\mathbf{P}^{-1}\mathbf{P}=\mathbf{I}$.
Thus the matrices $\mathbf{M}^{(k)}$ are both of type II
in (\ref{eq:affineLR}).

\medskip
The great generality of the above construction is accompanied by a
lack of specific information about the submatrices $\mathbf{A}_k$ 
 and the translation vector $\mathbf{b}_k$ in $\mathbf{M}^{(k)}$.
Below (theorem \ref{thm:IFSAttractor}) we establish when the above affine IFS
has a unique global connected attractor.
For this purpose we must show that all infinite products of $\mathbf{A}_0$ 
and $\mathbf{A}_1$ converge to the zero matrix, or equivalently, that their 
\textit{joint spectral radius} is less than one. 
Recall that the joint spectral radius of a set of matrices 
$\mathcal{A}$ is defined as
\[
\rho(\mathcal{A}):=\lim_{n\to\infty} \sup \{||\mathbf{A}_{i_1}\cdots \mathbf{A}_{i_n}||^\frac{1}{n}:
\mathbf{A}_{i}\in \mathcal{A} \},
\]
where $||\cdot||$ is any matrix norm.

In the 90's Daubechies and Lagarias \citep{Daubechies} defined the \textit{generalized spectral radius}, which was proved to be equal to the joint spectral radius for a bounded set of matrices \citep{BergerWang}.
The \textit{generalized spectral radius} is defined as:
\[
\rho(\mathcal{A}):=\lim_{n\to\infty}\sup \{ \rho(\mathbf{A}_{i_1}\dots \mathbf{A}_{i_n})^\frac{1}{n}, \mathbf{A}_{i}\in \mathcal{A} \}.
\]

We shall employ the following lemma:
\begin{lemma}\label{lem:JointSpectralRadius}
If $\mathbf{A}_0,\mathbf{A}_1$ can be simultaneously upper-triangularized, 
i.e., if there 
exists an invertible matrix $\mathbf{S}$ such that $\mathbf{S}^{-1}\mathbf{A}_0\mathbf{S}$ and $\mathbf{S}^{-1}\mathbf{A}_1\mathbf{S}$ 
are both upper-triangular, then $\rho(\mathbf{A}_0,\mathbf{A}_1) = max\{\rho(\mathbf{A}_0), \rho(\mathbf{A}_1)\}$.
\end{lemma}

\begin{proof}
The result follows immediately from examination of the diagonal entries of the product 
of two upper-triangular matrices. See \citep[Lemma 4.7]{HeilColella}.
\end{proof}

\begin{remark}
In \citep{BarnesleyVince} it is shown that a necessary and sufficient 
condition for a compact affine IFS in $\R^n$ to have a unique attractor 
is that $\rho(\mathcal{A}) <1$, where $\mathcal{A}$ is the set of matrices
corresponding to the linear part of the set of affine transformations
the IFS consists of.
This statement can be extended verbatim to our case of compact affine 
IFS in $\C^n$. 
\end{remark}

The next result, which is crucial to our analysis, exhibits
a conjugacy that transforms $\mathbf{L}(t)$ into a diagonal matrix 
and $\mathbf{R}(t)$ into an upper-triangular matrix, both of type
III in (\ref{eq:affineLR}); these matrices are determined explicitly.

\begin{lemma}\label{lma:BezierUpperTriangular}
For any $n\geqslant 0$, let $\mathbf{S}$ be the matrix of eigenvectors 
of the $(n+1)\times (n+1)$ matrix $\mathbf{L}(t)$. 
Then $\mathbf{S}$ is invertible, and
\begin{equation}\label{eq:BezierUpperTriangular}
\begin{array}{rcl}
\displaystyle (\mathbf{S}^{-1}\mathbf{L}(t)\mathbf{S})_{i,j}
&=& \displaystyle \binom{0}{i-j}t^i\\
\noalign{\vskip 10pt}
\displaystyle (\mathbf{S}^{-1}\mathbf{R}(t)\mathbf{S})_{i,j}
   &=&\displaystyle \binom{n-i}{n-j}t^{j-i}(1-t)^i
\end{array} \hskip 30pt i,j=0,\ldots,n.
\end{equation}
\end{lemma}

\begin{proof}
As noted above [see remarks after eq.~(\ref{eq:CastelSubd})], 
the $(n+1)$-dimensional matrix $\mathbf{L}(t)$ over 
$\Z[t]$ has the $(n+1)$ distinct eigenvalues $t^i$, $i=0,\ldots,n$.
Thus $\mathbf{L}(t)$ is diagonalizable via the similarity $\mathbf{S}$, and 
considering that $\binom{0}{i-j}=\delta_{i,j}$, the Kronecker delta,
the first formula in (\ref{eq:BezierUpperTriangular}) is proved.

Next we show that $\mathbf{S}_{i,j}=\binom{i}{j}$, by computing the
eigenvectors of $\mathbf{L}(t)$.
We shall use the identity \citep[p. 174]{GKP}
\begin{equation}\label{eq:BinomialIdentity}
\binom{i}{k}\binom{k}{j}=\binom{i}{j}\binom{i-j}{k-j}\qquad i,j,k\in\Z.
\end{equation}
Keeping in mind the expression for $\mathbf{L}(t)$ in (\ref{eq:CastelSubd}), 
and that $\mathbf{L}(t)_{i,k}=0$ for $k<0$ and $k>n$,
we compute the $i$-th component of the $j$-th eigenvector by summing over 
all integers $k$. We obtain
\begin{eqnarray*}
\sum_k B^i_k\binom{k}{j}&=&\sum_k\binom{i}{k}\binom{k}{j}t^k(1-t)^{i-k}\\
&=&\binom{i}{j}\sum_k\binom{i-j}{k-j}t^{(k-j)+j}(1-t)^{(i-j)-(k-j)}\\
&=&t^j\binom{i}{j}\sum_m\binom{n}{m}t^m(1-t)^{n-m}\\
&=&t^j\binom{i}{j}(t+1-t)^n=t^j\binom{i}{j},
\end{eqnarray*}
where the summation variable $m=k-j$ is unrestricted because $k$ is.
This computation shows that the $j$th column vector of $\mathbf{S}$ has components
$\binom{i}{j}, i=0,\ldots,n$, as desired.

\medskip
Next we show that $\mathbf{S}^{-1}_{i,j}=(-1)^{i+j}\binom{i}{j}$, i.e., 
$\mathbf{S}^{-1} \mathbf{S} = \mathbf{I}$. 
Again, Using (\ref{eq:BinomialIdentity}) and summing over $\Z$, we find:
\[
 \big(\mathbf{S}^{-1} \mathbf{S}\big)_{i,j} 
     = \sum_k(-1)^{i+k}\binom{i}{k}\binom{k}{j}
 = \binom{i}{j}(-1)^i\sum_k (-1)^{k}\binom{i-j}{k-j}.
\]
If $i<j$ then $\binom{i}{j}=0$, and the rightmost expression vanishes.
The same happens if $j<i$, because in this case, from 
the binomial theorem, we have 
$$
\sum_k(-1)^k\binom{i-j}{k-j}=(-1)^j\sum_m\binom{i-j}{m}(-1)^m=(-1)^j(-1+1)^{i-j}=0.
$$
Finally, for $i=j$ the only 
non-zero term in the sum corresponds to $k=i$, giving 
$\big(\mathbf{S}^{-1} \mathbf{S}\big)_{i,i}=\binom{i}{i}(-1)^{2i}=1$, as desired.

\medskip
We have shown that the matrix product 
$\mathbf{S}^{-1}\mathbf{L}(t) \mathbf{S}$ is diagonal.
It remains to show that 
$\mathbf{S}^{-1}\mathbf{R}(t) \mathbf{S}$ is upper-triangular. 
From the above and (\ref{eq:CastelSubd}), we find:\\

\noindent
\scalebox{0.81}{
$
\mathbf{S}^{-1}\mathbf{R}(t) \mathbf{S} =
\begin{pmatrix}
 1          & 0                         & \ldots & 0 \\
 (-1)^1     & (-1)^{1+1} \binom{1}{1}   & \ldots & 0 \\
 \vdots     & \vdots                    & \vdots   & 0 \\
 (-1)^{n} &(-1)^{n+1}\binom{n}{1} & \ldots & (-1)^{n+n}\binom{n}{n} \\
\end{pmatrix}
\begin{pmatrix}
B^n_0(t) & B^n_1(t) & \ldots & B^n_n(t)\\
0 & B_0^{n-1}(t) & \ldots & B_{n-1}^{n-1}(t) \\
\vdots & \vdots & \vdots & \vdots \\
0 & 0 & \ldots & B_0^0(t)\\
\end{pmatrix}
\begin{pmatrix}
 1      & 0             & \ldots    & 0 \\
 1      & \binom{1}{1}  & \ldots    & 0  \\
 \vdots & \vdots        & \vdots    & 0 \\
 1      & \binom{n}{1}  & \ldots    & \binom{n}{n} \\
\end{pmatrix}
$.
}
\\

The explicit expression for the $(i,j)$-component $T^{(n)}_{i,j}$ of 
the matrix $\mathbf{S}^{-1}\mathbf{R}(t)\mathbf{S}$ is thus given by
\begin{eqnarray}\label{eq:T}
T^{(n)}_{i,j}
&=& \sum_{l=0}^{n} \sum_{k=0}^{n} (-1)^{i+l} \binom{i}{l} 
  B^{n-l}_{k-l}(t)\binom{k}{j}\nonumber\\
&=& \sum_{l} \sum_{k} (-1)^{i+l}\binom{i}{l}\binom{n-l}{k-l}
\binom{k}{j} t^{k-l} (1-t)^{n-k},
\end{eqnarray}
where the summation ranges have been extended to $\Z$ because the
binomial coefficients provide the stated restrictions.
The above expression is a polynomial in $\Z[t]$.

Let us consider the sequence of matrices over $\Z[t]$:
\begin{equation}\label{eq:hatT}
\hat T^{(n)}_{i,j}=\binom{n-i}{n-j}t^{j-i}(1-t)^i\qquad 
n\geqslant 0,\quad i,j=0,\ldots,n.
\end{equation}
We have $\hat T^{(0)}=(1)$, and for all $n>0$, $\hat T^{(n)}$ is upper triangular, 
with diagonal elements $\hat T^{(n)}_{i,i}=(1-t)^i$ (the eigenvalues of 
$\mathbf{R}(t)$).

We begin to show that the matrices $\hat T^{(n)}$ satisfy the identity
\begin{equation}\label{eq:Tidentity}
\hat T^{(n+1)}_{i+1,j+1}+\hat T^{(n+1)}_{i,j+1}
  = \hat T^{(n)}_{i,j+1}+\hat T^{(n)}_{i,j}\qquad n\geqslant 0,\quad i,j=0,\ldots,n.
\end{equation}
Here we have included the value $j=n$ because from (\ref{eq:hatT}) we have
$T^{(n)}_{i,n+1}=0$ for all $i$ (the lower index of the binomial coefficient is
negative), even though this expression no longer represents a matrix element.

The left-hand side of (\ref{eq:Tidentity}) is given by
\begin{eqnarray*}
\hat T^{(n+1)}_{i+1,j+1}+\hat T^{(n+1)}_{i,j+1}
&=&t^{j-i}(1-t)^i\left[\binom{n-i}{n-j}(1-t)+\binom{n+1-i}{n-j}t\right]\\
&=&t^{j-i}(1-t)^i\left[\binom{n-i}{n-j}+\binom{n-i}{n-j-1}t\right],
\end{eqnarray*}
where we have used the binomial formula \citep[p. 174]{GKP}
\begin{equation}\label{eq:BinomialRecursion}
\binom{n}{k}=\binom{n-1}{k}+\binom{n-1}{k-1}\qquad n,k\in\Z.
\end{equation}
For the right-hand side of (\ref{eq:Tidentity}), we have
\begin{eqnarray*}
\hat T^{(n)}_{i,j+1}+\hat T^{(n)}_{i,j}
&=&t^{j-i}(1-t)^i\left[\binom{n-i}{n-j-1}t+\binom{n-i}{n-j}\right],
\end{eqnarray*}
which establishes (\ref{eq:Tidentity}).

Next we show 
---by double induction on $n$ and $i$---
that the entries of $\hat T^{(n)}$ are uniquely determined 
by its first row $(i=0$) and column $(j=0)$.
If $n=0$ this is trivially true, so assume that for some $n\geqslant0$,
the matrix $\hat T^{(n)}$ is known, and the first row and columns of
$\hat T^{(n+1)}$ are also known. Then, from (\ref{eq:Tidentity}) we have
that for all $j$ the sequence of equation
$$
\hat T^{(n+1)}_{i+1,j+1}=-\hat T^{(n+1)}_{i,j+1}
  + \hat T^{(n)}_{i,j+1}+\hat T^{(n)}_{i,j},\qquad i=0,\ldots,n
$$
may be solved recursively from the knowledge of the first term 
$\hat T^{(n+1)}_{0,j+1}$.
One verifies that the above procedure accounts for all the missing 
entries of $\hat T^{(n+1)}$. 
[Recall the remark following (\ref{eq:Tidentity}).]
\medskip

We shall now prove that $T^{(n)}_{i,j}$ in (\ref{eq:T}) is equal
to $\hat T^{(n)}_{i,j}$,
by showing that $T^{(0)}=\hat T^{(0)}$, and that for $n>0$ 
the matrices $T^{(n)}$ and $\hat T^{(n)}$ share the first row 
and column, and satisfy the same recursion formula (\ref{eq:Tidentity}).

We begin by computing the first row of $T^{(n)}$.
For $i=0$, the term $\binom{i}{l}$ in (\ref{eq:T})
forces $l=0$, and hence using (\ref{eq:BinomialIdentity}) and the
binomial theorem, we find
\begin{eqnarray}\label{eq:Trow}
T^{(n)}_{0,j}
&=& \sum_{k}\binom{n}{k}\binom{k}{j} t^{k} (1-t)^{n-k}\nonumber\\
&=& \binom{n}{j} \sum_{k}\binom{n-j}{k-j} t^{k+j-j} (1-t)^{n-k+j-j}\nonumber\\
&=& \binom{n}{j}t^j \sum_{m}\binom{n-j}{m} t^{m} (1-t)^{n-j-m}\nonumber\\
&=&\binom{n}{j}t^j,\qquad n\geqslant0,\quad i=0,\ldots,n,
\end{eqnarray}
which is (\ref{eq:hatT}) for $i=0$, as desired. 
This also shows that $T^{(0)}=(1)=\hat T^{(0)}$.

Next we compute the first column of $T^{(n)}$.
For $j=0$, the binomial coefficient in 
(\ref{eq:hatT}) is zero unless also $i=0$, in which case it is
equal to 1. We find
\begin{eqnarray}\label{eq:Tcolumn}
T^{(n)}_{i,0}
&=& \sum_{l}(-1)^{i+l}\binom{i}{l}
   \sum_k\binom{n-l}{k-l} t^{k-l} (1-t)^{n-l+(k-l)}\nonumber \\
&=& (-1)^i\sum_{l}(-1)^{l}\binom{i}{l}=\delta_{i,0},
\qquad n\geqslant 0,\quad i=0,\ldots,n.
\end{eqnarray}
Indeed for $i>0$ the sum is zero, from the binomial theorem, 
while for $i=0$, we have $l=0$, hence $T^{n}_{0,0}=1$ (as above),
which again agrees with (\ref{eq:hatT}) for $j=0$.

It remains to verify that $T$ satisfy (\ref{eq:Tidentity}).
Using (\ref{eq:BinomialIdentity}) twice, we find
\begin{eqnarray*}
T^{(n+1)}_{i+1,j+1}&=&\sum_l(-1)^{i+1+l}\binom{i+1}{l}
   \sum_k\binom{n+1-l}{k-l}\binom{k}{j+1}t^{k-l}(1-t)^{n+1-k}\\
&=&-\sum_l(-1)^{i+l}\binom{i}{l}\sum_k\binom{n+1-l}{k-l}
       \binom{k}{j+1} t^{k-l}(1-t)^{n+1-k}\\
&&+\sum_l(-1)^{i+(l-1)}\binom{i}{l-1}\sum_k\binom{n-(l-1)}{k-1-(l-1)}
       \left[\binom{k-1}{j+1}+\binom{k-1}{j}\right] \\
&&\hskip 50pt   \times\, t^{k-l}(1-t)^{n-(k-1)}\\
&=&-T^{(n+1)}_{i,j+1} + T^{(n)}_{i,j+1} + T^{(n)}_{i,j}.
\end{eqnarray*}
This establishes the identity
$$
T^{(n+1)}_{i+1,j+1}+T^{(n+1)}_{i,j+1}
  = T^{(n)}_{i,j+1}+T^{(n)}_{i,j},
$$
as desired. We remark that this recurrence holds also for $j=n$, 
because in (\ref{eq:T}) we have $\binom{k}{n+1}=0$ for all $k$ 
for which $\binom{i}{l}\binom{n-l}{n-k}\not=0$, as easily verified. 
Hence $T^{(n)}_{i,n+1}=0=\hat T^{(n)}_{i,n+1}$.

We have shown that for all $n\geqslant 0$, the matrices $T^{(n)}$ 
and $\hat T^{(n)}$ have the same first row and column, and their entries 
satisfy the same recursion formula. 
Hence $T^{(n)}=\hat T^{(n)}$, for all $n$.

The proof of the lemma is complete.
\end{proof}

From the proof of lemma \ref{lma:BezierUpperTriangular} we obtain ---as a bonus--- the following identities 
for the double sums $T^{(n)}$, whose validity for $\hat T^{(n)}$ 
in (\ref{eq:hatT}) is verified at once. 
\begin{corollary}
Let $T^{(n)}_{i,j}$ be the sum (\ref{eq:T}).
Then, for $n\geqslant 0$ and $j=0,\ldots,n$:
\begin{equation*}\label{eq:Identities}
\begin{array}{rcll}
T^{(n)}_{i+1,j+1}&=&(1-t)\;T^{(n)}_{i,j}-t\;T^{(n)}_{i+1,j}
&\quad i=0,\ldots,n-1\\
\noalign{\vskip 5pt}
T^{(n+1)}_{j+1,j+1}&=&(1-t)\;T^{(n)}_{i,j}
&\quad  i=0,\ldots,n.
\end{array}
\end{equation*}
\end{corollary}

For our next result we will need the following theorem 
\citep[Chapter VIII, Theorem 2.1]{FractalsEverywhere}.

\begin{theorem}\label{thm:Connectedness}
Let $\{f_0,f_1\}$ be a hyperbolic IFS with attractor $\mathcal{A}$. 
Let $f_0$ and $f_1$ be one-to-one on $\mathcal{A}$. If
\begin{equation*}
    f_0(\mathcal{A}) \cap f_1(\mathcal{A})=\emptyset,
\end{equation*}
then $\mathcal{A}$ is totally disconnected. If
\begin{equation*}
     f_0(\mathcal{A}) \cap f_1(\mathcal{A})\neq \emptyset,
\end{equation*}
then $\mathcal{A}$ is connected.
\end{theorem}

\begin{figure}[ht]
\begin{center}
\includegraphics[width=0.4\textwidth]{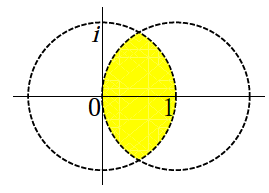}
\caption{The domain of $t$, over which the IFS $\{\mathbb{C};f_1,f_2\}$ is hyperbolic, is the intersection of two open discs $|t|<1$ and $|1-t|<1$.}\label{fig:ComplexParam_t}
\end{center}
\end{figure}

\begin{figure}[ht]
\begin{center}
\includegraphics[width=0.4\textwidth]{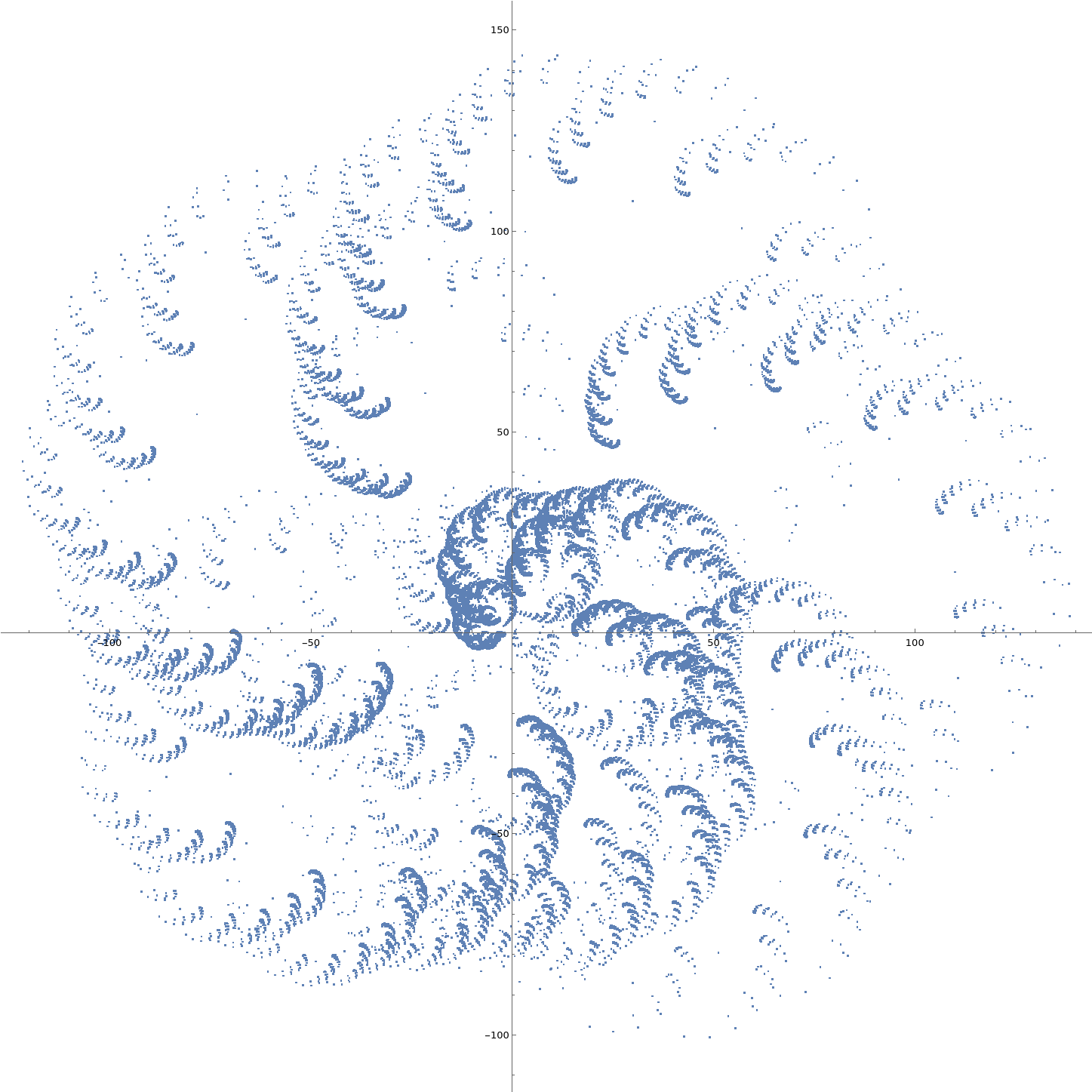}
\includegraphics[width=0.4\textwidth]{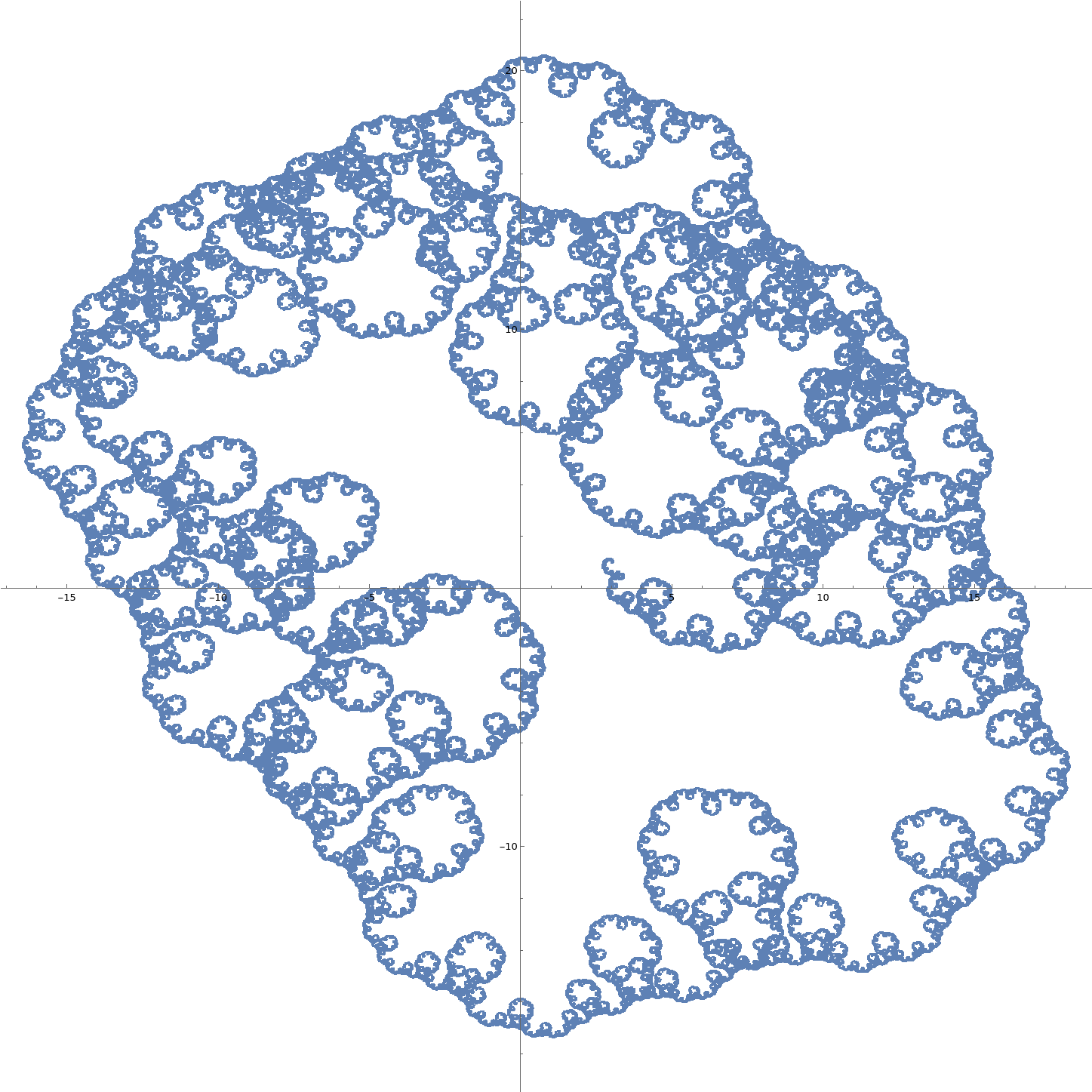}\\
\includegraphics[width=0.4\textwidth]{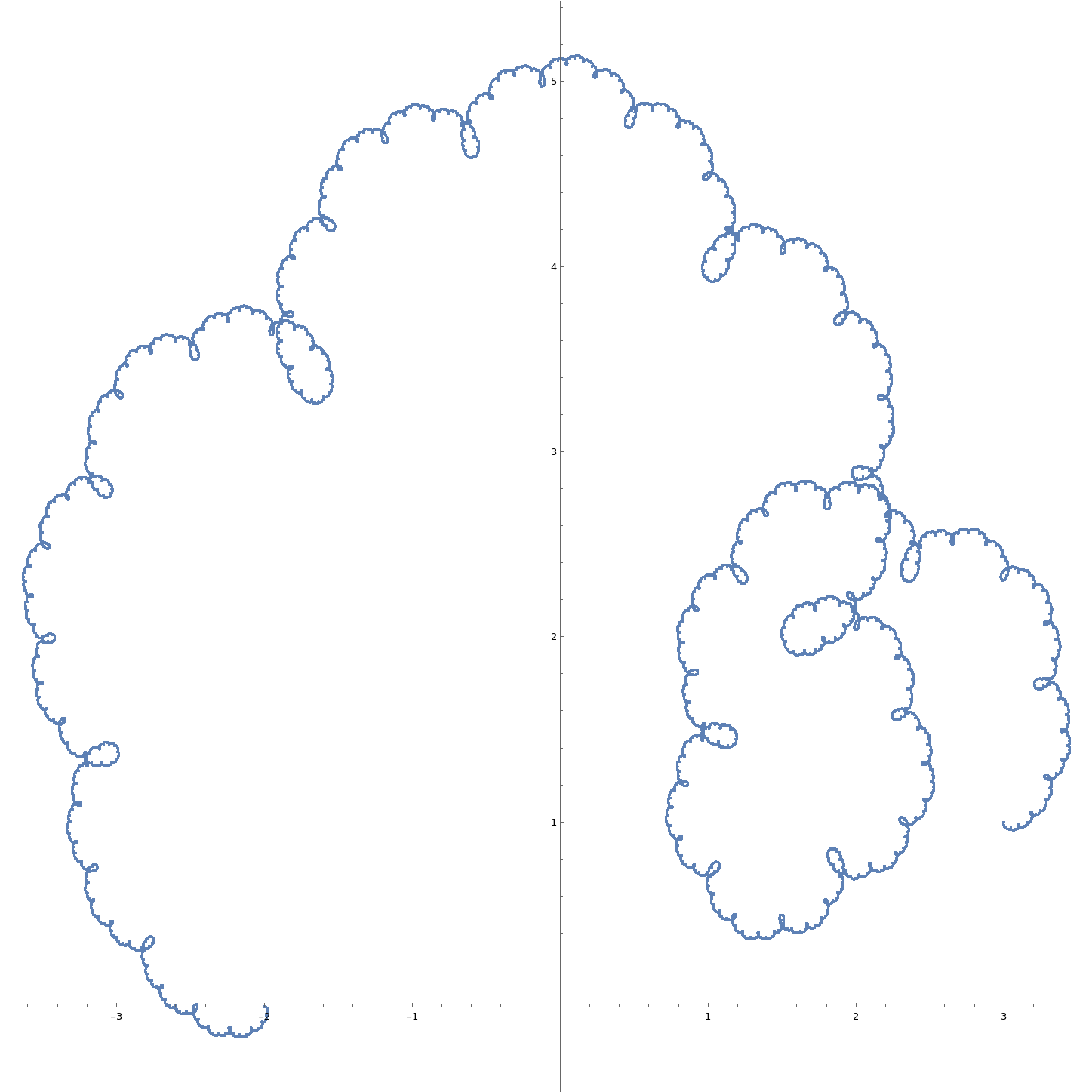}
\includegraphics[width=0.4\textwidth]{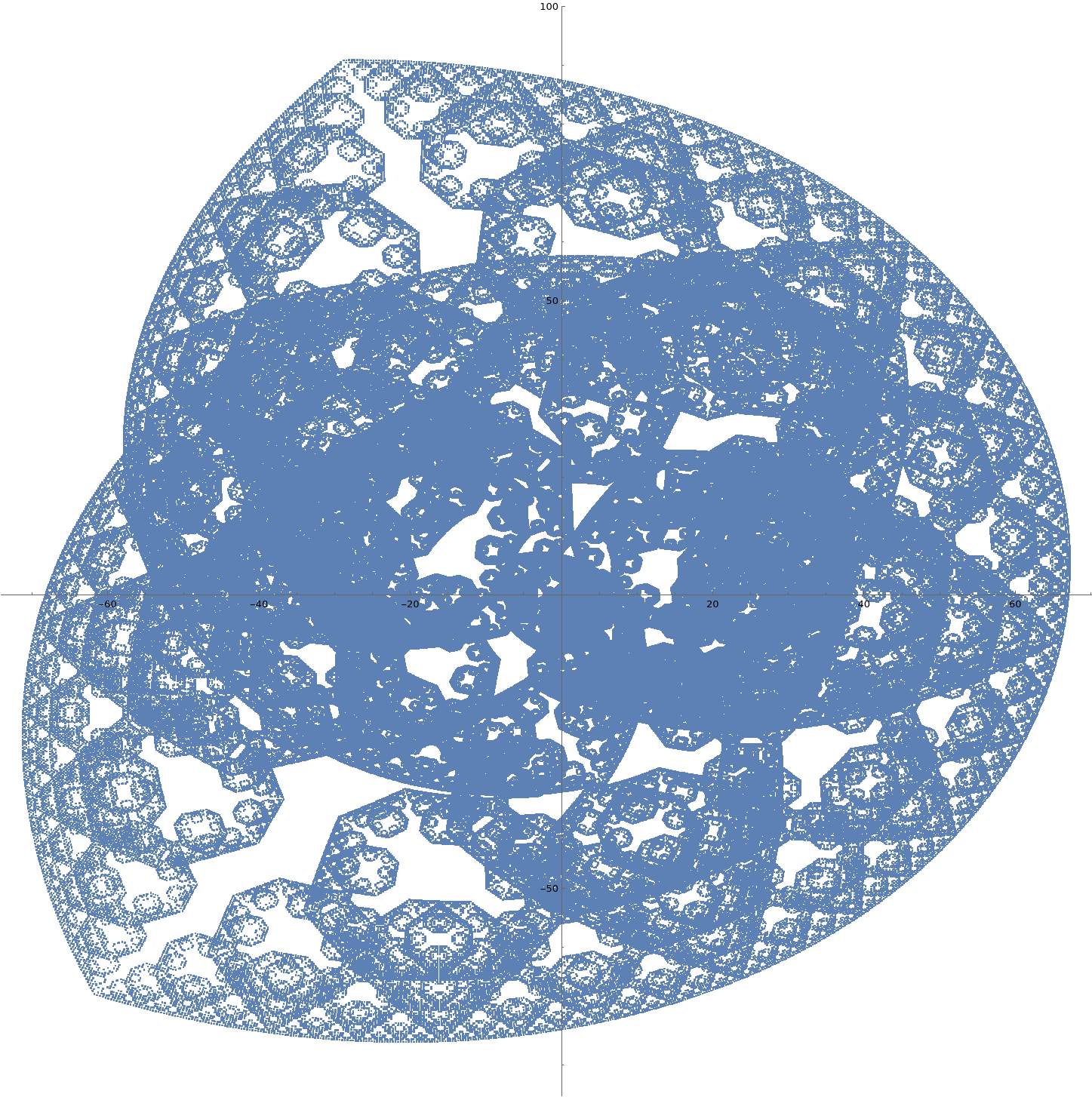}
\caption{Attractor of the IFS for the B\'ezier curve of degree 6 with control points $(-2,-1+i,2 i,1+i,1,2+i,3+i)^\top$ and complex parameters (from top to bottom and left to right): $t = 0.1 + 0.3 i$, $t = 0.4 + 0.4 i$, $t = 0.5 + 0.25 i$, and $t = 0.5 + 0.5 i$. We show the state of the system after 20 iterations. To visualize the higher dimensional attractor, we project this attractor to the complex plane. Note that the fixed point of $f_0$ is the first row of matrix $\mathbf{P}$ of (\ref{eq:BezierIFS}), i.e., $(-2,1,0,0,0,0,1)$, while the fixed point of $f_1$ is 
$(3+i,0,0,0,0,0,1)$ --- the last row of $\mathbf{P}$.}\label{fig:sexticBezier}
\end{center}
\end{figure}

We now show that the global attractor of the de Casteljau IFS for a complex B\'ezier 
curve is unique and connected. This theorem is the main result of this section.

\begin{theorem}\label{thm:IFSAttractor}
The affine IFS associated to the subdivision algorithm for a B\'ezier curve 
[as in (\ref{eq:BezierIFS}) or (\ref{eq:BezierUpperTriangular})] has a unique and 
connected attractor for any parameter $t\in\C$ such that $|t|<1\wedge|1-t|<1$.
\end{theorem}

\begin{proof}
From lemma \ref{lma:BezierUpperTriangular} the similarity 
$\mathbf{S}$ transforms the de Casteljau matrices $\mathbf{L}(t)$ and $\mathbf{R}(t)$ 
to the form (\ref{eq:affineLR}) III, moreover their $n\times n$ submatrices 
$\mathbf{A}_0$ and $\mathbf{A}_1$ are, respectively, diagonal and upper-triangular.
Thus also $\mathbf{A}_0$ and $\mathbf{A}_1$ are simultaneously triangularizable, independent of the chosen representation from the list (\ref{eq:affineLR}), since the change of representation can be achieved by 
a reflection $(i,j) \leftrightarrow (n-i,n-j)$ and/or a transposition  ---cf.~discussion following table (\ref{eq:affineLR}).

From lemma \ref{lem:JointSpectralRadius} we then have that
\[
 \rho(\mathbf{A}_0,\mathbf{A}_1) 
   =\max\{\rho(\mathbf{A}_0),\rho(\mathbf{A}_1)\} = \max\{|t|,|1-t|\}.
\]
and hence in the domain $t\in\mathbb{C},\,|t|<1\wedge|1-t|<1$, all infinite products
of the matrices $\mathbf{A}_0$ and $\mathbf{A}_1$ converge to the zero
matrix, which ensures the contractivity of the IFS (in some metric). 
From Banach theorem \ref{thm:BanachFixedPoint} this IFS has a unique 
global attractor.

For computations we shall represent the IFS with the matrices 
\begin{equation}\label{eq:ltM}
\bM^{(0)}_{i,j}=\binom{0}{i-j}t^i
\hskip 40pt
\bM^{(1)}_{i,j}=\binom{n-j}{n-i}t^{i-j}(1-t)^j
\qquad i,j=0,\ldots,n.
\end{equation}
of type IV in (\ref{eq:affineLR}); these are the transposes of the 
matrices (\ref{eq:BezierUpperTriangular}), and act on the right on
$(n+1)$-dimensional column vectors representing homogeneous coordinates 
in $\C^n$.

Let $\{f_0,f_1\}$ be the corresponding affine IFS. 
The maps $f_0$ and $f_1$, being affine, are one-to-one on the attractor $\cA$, 
and in the assumed parameter range they are contractions, so their respective
fixed points $z_0^*$ and $z_1^*$ belong to $\cA$.

We claim that $z_0^*=(1,0^m)^\top$ and 
$z_1^*=(\binom{n}{0},(\binom{n}{1},\ldots,\binom{n}{n})^\top$. 
The former is clear, since $\bM^{(0)}$ is diagonal. As to the latter, 
using (\ref{eq:BinomialIdentity}) and the binomial theorem, we find
\begin{eqnarray*}
(\bM^{(1)}z_1^*)_i&=&\sum_j\binom{n-j}{n-i}\binom{n}{j}t^{i-j}(1-t)^j\\
&=&\sum_j\binom{n}{n-j}\binom{n-j}{n-i}t^{i-j}(1-t)^j\\
&=&\binom{n}{n-i}\sum_j\binom{i}{i-j}t^{i-j}(1-t)^j=\binom{n}{i},
\end{eqnarray*}
as desired.
A straightforward calculation shows that
$\bM^{(0)}z_1^*=\bM^{(1)}z_0^*$ are both equal to
$$
\bigl(\binom{n}{0},t\binom{n}{1},\ldots,t^n\binom{n}{n}\bigr)^\top.
$$
The above shows that $f_0(\mathcal{A}) \cap f_1(\mathcal{A})\neq \emptyset$,
and hence $\cA$ is connected by theorem \ref{thm:Connectedness}. 
\end{proof}

The domain of $t$ over which the IFS has a unique global attractor 
is displayed in figure \ref{fig:ComplexParam_t},
while in figure \ref{fig:sexticBezier} we show examples of the unique global
attractor of the IFS, for specific parameter values.

\begin{remark}
Tsianos and Goldman in \citep[Lemma 3.4]{Tsianos2011} establish the same domain of complex parameter $t$ over which the lengths of the polygons generated by recursive subdivision at $t$ converge to zero. However, the starting set of their algorithm is a control polygon, whereas our starting set is arbitrary, and in particular not necessarily connected; thus their result is not readily applicable in our context.
\end{remark}

\section{The Takagi curve}\label{sec:Takagi}
The Takagi curve is a graph of the Takagi (or Blancmange) function.
Some of its remarkable properties are described in the surveys
\citep{Allaart, Lagarias}. The Takagi function is defined by
\begin{equation*}
    \rT(x)=\sum_{n=0}^{\infty} \frac{\sigma(2^nx)}{2^n}\qquad x\in [0,1]
\end{equation*}
where $\sigma(x)=\min_{n\in\mathbb{Z}} |x-n|$, that is, 
$\sigma(x)$ is the distance from $x$ to the nearest integer.

The Takagi curve has infinite length over any nonempty open interval in 
$[0,1]$ \citep[Theorem 11.4]{Lagarias}, and it is self-similar in the
sense that it satisfies a dyadic self-similarity equation \citep[Theorem 4.1]{Lagarias}. 
However, if we adopt Mandelbrot's definition of a fractal as a set whose 
Hausdorff dimension is strictly greater than its topological dimension 
\citep[p. 15]{Mandelbrot}, then the Takagi curve is not a fractal
since its topological dimension and its Haussdorff dimension 
(as a subset of $\R^2$) are both equal to 1 \citep[Theorem 11.2]{Lagarias}. 

We shall prove that the Takagi curve represents the linearisation 
of the attractor of an IFS for a B\'ezier curve with complex
parameter, as the parameter's imaginary part goes to zero.
We consider the de Casteljau algorithm for the B\'ezier curve with two
control points, $0$ and $1$, in the complex plane. For a real parameter
$t$ the attractor $\mathcal{A}$ is the line segment connecting these two points,
but for complex $t$ the attractor is non-trivial.
The IFS is represented in homogeneous coordinates as the pair of complex 
matrices of type I in (\ref{eq:affineLR})
\begin{equation}\label{eq:M01}
\mathbf{M}^{(0)} = \left(\begin{array}{cc} t& 0\\0 &1\end{array}\right)
\hskip 30pt
\mathbf{M}^{(1)} = \left(\begin{array}{cc} 1-t& t\\0 &1 \end{array} \right).
\end{equation}
The above matrices correspond to the pair of complex maps
\begin{equation}\label{eq:f}
f^{(0)}(z)=tz\qquad f^{(1)}(z)=(1-t)z+t.
\end{equation}
\medskip
We shall consider parameters of the form $t=1/2+\mathrm{i}\beta$, 
with real $\beta$. 
From theorem \ref{thm:IFSAttractor},
for $|\beta|<\sqrt{3}/2$ the IFS is hyperbolic with a connected 
attractor $\mathcal{A}=\mathcal{A}(\beta)$. 
We then define the scaling map
\begin{equation}\label{eq:g}
    \rg: \mathbb{C}\times\mathbb{R} \to \mathbb{C} \qquad \rg(z,\beta)
    =\Re(z)+ \beta\,\ri\,\Im(z)
\end{equation}
together with the scaled attractor of the IFS:
\begin{equation}\label{eq:A^*}
\mathcal{A}^*(\beta)=\rg(\mathcal{A}(\beta),(2\beta)^{-1}).
\end{equation}
The properties of the scaled attractor depend strongly on $\beta$. 
In Figures \ref{fig:Ccurve}--\ref{fig:Tak48} we plot 
$\mathcal{A}^*(\beta)$ of (\ref{eq:A^*}) for various values of $\beta$, 
and the graph $\cT$ of the Takagi function for comparison. 
In all cases, the IFS is iterated 15 times, with
the unit interval as the initial set.

\begin{figure}[ht]
\begin{center}
\includegraphics[width=0.7\textwidth]{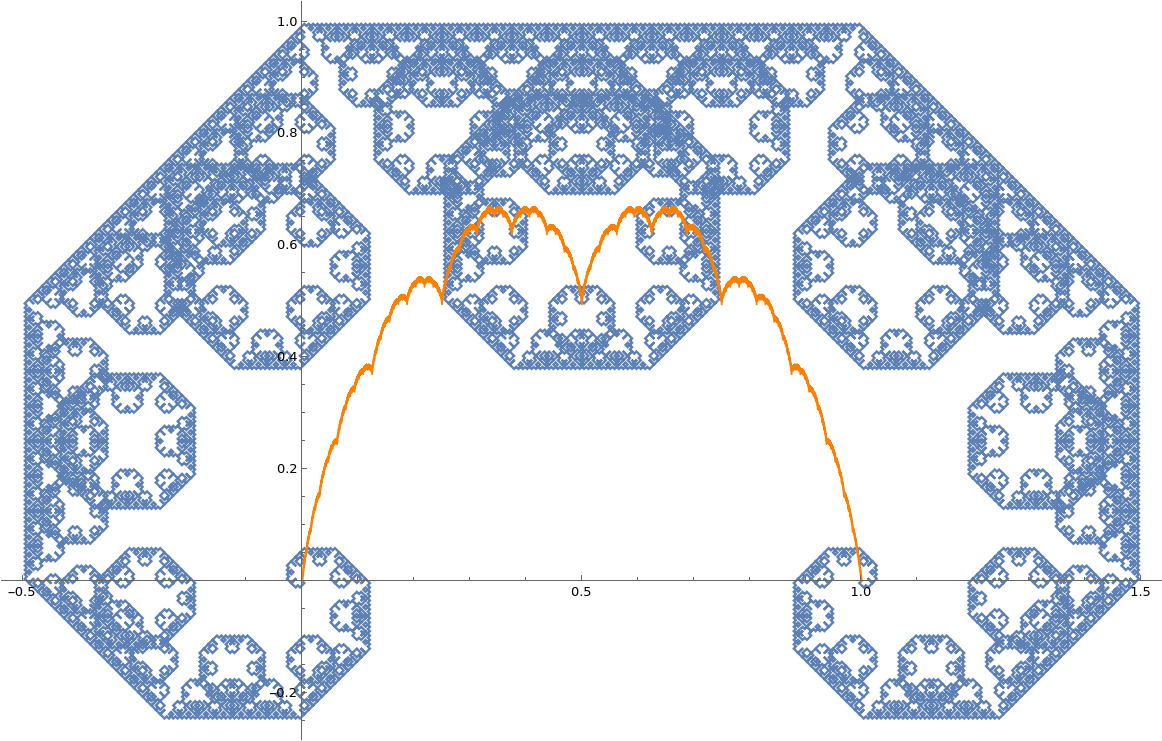}
\caption{The blue curve is $\mathcal{A}^\ast(\beta)$, 
which for $\beta=\frac{1}{2}$ becomes the L\'evy C curve. 
The orange curve is the graph $\cT$ of the Takagi function 
$\rT$.}\label{fig:Ccurve}
\end{center}
\end{figure}

\begin{figure}[ht]
\begin{center}
\includegraphics[width=0.49\textwidth]{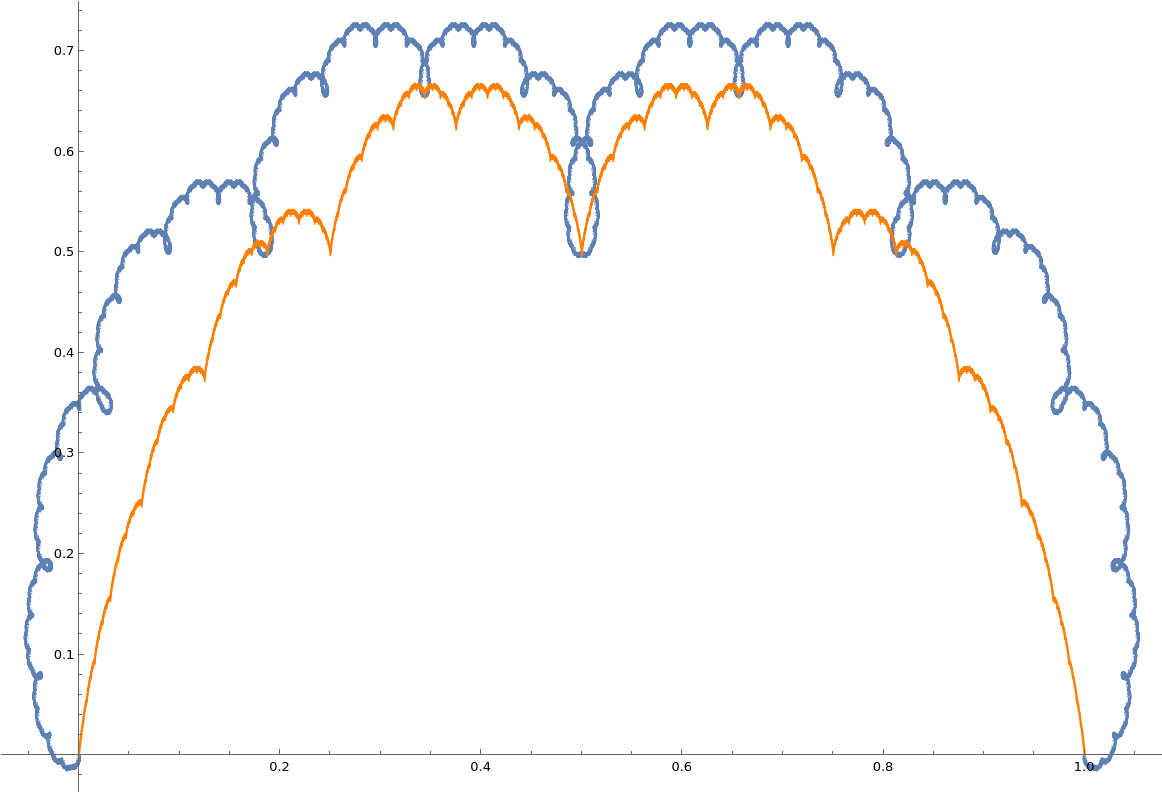}
\includegraphics[width=0.49\textwidth]{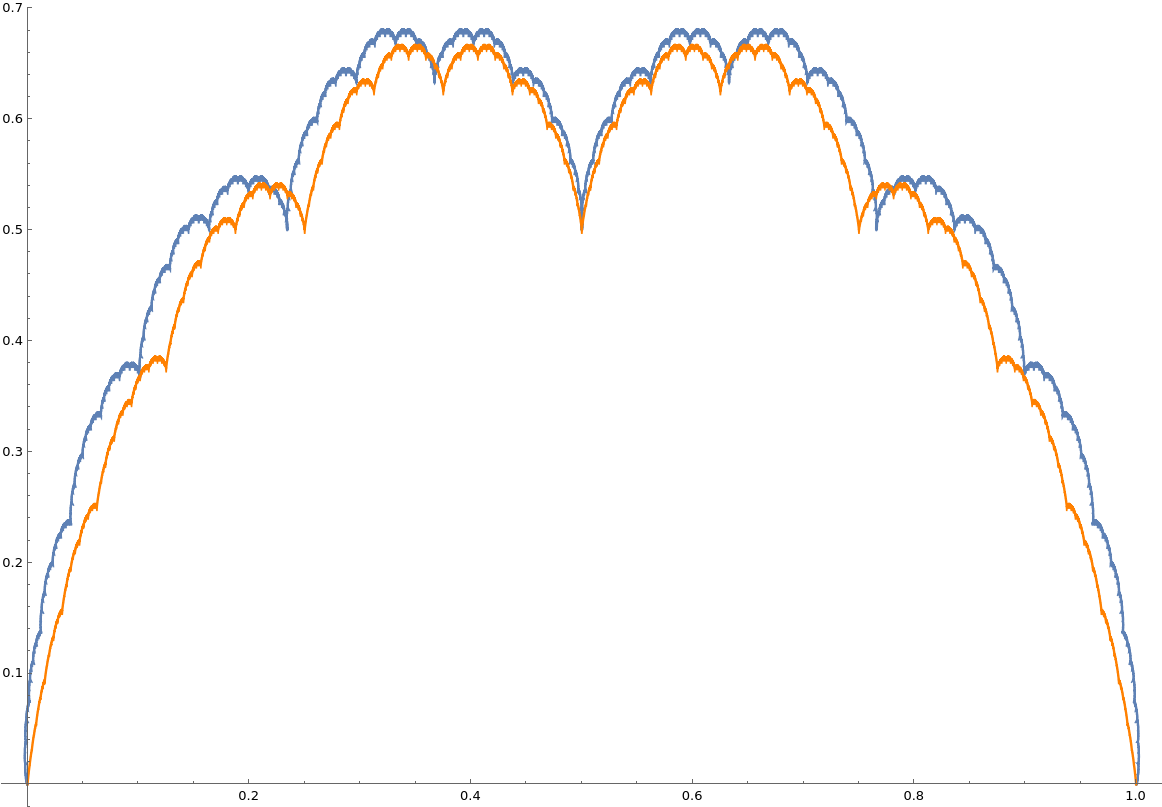}
\caption{Left: the scaled attractor $\mathcal{A}^\ast(\beta)$ 
for $\beta=\frac{1}{4}$ (blue) and the Takagi curve (orange).
Right: the same for $\beta=\frac{1}{8}$.
For $\beta=\frac{1}{4}$, the attractor features crunodes; 
these crunodes appear to have turned into cusps for $\beta=\frac{1}{8}$,
although this is not the case (see remark after theorem 
\ref{thm:Takagi}).}\label{fig:Tak48} 
\end{center}
\end{figure}

\medskip

We begin our analysis of this class of phenomena.
Let $\Omega=\{0,1\}^\N$ be the space of infinite binary sequences.
Any $d\in\Omega$ represents the digits of a real number 
$x^{(d)}\in [0,1]$. 
Since the elements of $\Z[1/2]$ are precisely the real numbers 
that admit two distinct binary representations, for uniqueness
we shall stipulate that these digits must eventually be all 
zeros rather than all ones, with the exception of $d=(1^\infty)$.
With this provision the map 
\begin{equation}\label{eq:pi}
\pi:\Omega \to [0,1]\hskip 30pt
d\mapsto x^{(d)}=\sum_{k=1}^\infty\frac{d_k}{2^k}
\end{equation}
is a bijection, and we shall speak of convergence and measure 
in $\Omega$ in terms of the distance and Lebesgue measure in the 
unit interval. Note that $d$ is eventually periodic iff $x^{(d)}\in\Q$. 

Next we define the operator that reverses the first $n$ terms of a sequence:
For every $d=(d_1,d_2,\ldots)\in\Omega$ and $n\in\N$ we let
\begin{equation}\label{eq:revd}
 \revd(n):=(d_n,d_{n-1},\ldots,d_1,0^\infty).
\end{equation}
From our convention on digits it follows that 
$\lim_{n\to\infty}\revd (n)=(0^\infty)$ iff
$x^{(d)}$ is a dyadic rational\footnote{A rational whose denominator 
is a power of $2$. According to our convention, it has only finitely 
many non-zero binary digits.}.

For any sequence $d$ we define
\begin{equation}\label{eq:ru}
r_n^{(d)}=\sum_{k=1}^n\frac{d_k}{2^k}
\hskip 30pt
u_n^{(d)}=\sum_{k=1}^nd_k.
\end{equation}
Thus $r_n$ is the rational number whose binary digits are $d(n)$
and $u_n^{(d)}$ is the number of ones in $d(n)$.
Plainly, 
\begin{equation}\label{eq:limit}
\lim_{n\to\infty}r_n^{(d)}=x^{(d)}.
\end{equation}

Further, given $d$, we define the matrix product
\begin{equation}\label{eq:M_n}
   \mathbf{M}_n(d)=\mathbf{M}^{(d_1)}\mathbf{M}^{(d_2)}\dots\mathbf{M}^{(d_n)},
\end{equation}
which corresponds to the affine map
\begin{equation}\label{eq:f^d}
f_n^{(d)}=f^{(d_1)}\circ f^{(d_2)}\circ\cdots\circ f^{(d_n)}.
\end{equation}
This map performs $n$ iterates, selecting the map $f^{(k)}$
according to the digit $d_k$ in reverse order, that is, the 
symbolic sequence that defines $f^{(d)}$ is $\revd$.
(The reason for this unconventional choice will become clear below.)

The map $f_n^{(d)}$, being a composition of affine maps, is affine,
and an easy induction shows that it has the form
\begin{equation}\label{eq:f_n}
    f_n^{(d)}(z)=U_n^{(d)}z+Z_n^{(d)}\qquad U_n^{(d)}=(1-t)^{u_n}t^{n-u_n},
\end{equation}
where $Z_n^{(d)}$ is a polynomial in $t$.

We now specialise to parameters of the form $t=1/2+\mathrm{i}\beta$, 
with real $\beta$.
We choose $z=0\in\mathcal{A}$ as the initial point of the orbit. 
From (\ref{eq:f_n})  we obtain
\begin{equation}\label{eq:Z_n}
f_n^{(d)}(0)=Z_n^{(d)},
\end{equation}
and from (\ref{eq:f}) we see that $Z_n^{(d)}$ is a polynomial in
$\ri\beta$ with coefficients in $\Z[1/2]$, the set of rational
numbers whose denominator is a power of 2.
Writing
\begin{equation}\label{eq:Zpoly}
Z_n^{(d)}=\sum_{k\geqslant 0} a_{k,n}^{(d)}(\ri\beta)^{k}
=X_n^{(d)}+\mathrm{i}Y_n^{(d)}, \qquad a_{k,n}\in\Z[1/2],
\end{equation} 
we find that $X_n^{(d)}$ and $Y_n^{(d)}$ are real polynomials
that correspond, respectively, to the even and odd powers of $\beta$.

\begin{lemma}\label{lma:Zpoly}
The polynomials $Z_n^{(d)}$ of (\ref{eq:Z_n}) satisfy the 
recursion relation
\begin{equation}\label{eq:Zrecursion}
Z_0^{(d)}=0\hskip 30pt
Z_{n+1}^{(d)}=\begin{cases} Z_{n}^{(d)} & d_{n+1}=0\\
                      Z_{n}^{(d)}+W_{n}^{(d)}& d_{n+1}=1
        \end{cases}
\end{equation}
where
\begin{equation}\label{eq:W}
W_n^{(d)}(\beta)=\sum_{m\geqslant 0} w_{m,n}\beta^m
\hskip 30pt
w_{m,n}=\frac{\mathrm{i}^m}{2^{n+1-m}}
     \sum_{k=0}^m(-1)^k\binom{u_n}{k}\binom{n+1-u_n}{m-k}.
\end{equation}
\end{lemma}

\begin{proof} 
We proceed by induction on $n$. 
For $n=0$ from (\ref{eq:f^d}) we have
$$
Z^{(d)}_1=\begin{cases}0 & d_1=0\\ 
\frac{1}{2}+\ri\beta & d_1=1,
\end{cases}
$$
so formula (\ref{eq:Zrecursion}) holds in the case $d_1=0$.
If $d_1=1$, then the sums (\ref{eq:W}) with $n=u_n=0$ are restricted 
to $k=0$ and $m=0,1$, for which all binomial coefficients are 1.
We find $w_{0,0}=1/2$ and $w_{1,0}=\ri$, in accordance with 
(\ref{eq:Zrecursion}).

Assume now that (\ref{eq:Zrecursion}) and (\ref{eq:W}) hold for some 
$n\geqslant 1$.
We have two cases. If $d_{n+1}=0$, then from (\ref{eq:f}) we 
have $f_{d_{n+1}}(0)=0$, that is, the first iterate of
the map doesn't change the initial point. Hence,
from (\ref{eq:Z_n}) we have $Z_{n+1}^{(d)}=Z_n^{(d)}$, so that
\begin{equation}\label{eq:d=0}
a_{k,n+1}^{(d)}=a_{k,n}^{(d)} \quad k=0,1,\ldots\hskip 30pt (d_{n+1}=0),
\end{equation}
in accordance with (\ref{eq:Zpoly}) and (\ref{eq:Zrecursion}).

If $d_{n+1}=1$, then we have $f_{d_{n+1}}(0)=t$
and from (\ref{eq:f_n}) we have
\begin{equation}\label{eq:U}
Z_{n+1}^{(d)}=f_{n+1}^{(d)}(0)=f_n^{(d)}(t)=U_n^{(d)}t+Z_n^{(d)}.
\end{equation}
Letting $W_{n}^{(d)}=U_{n}^{(d)}t$ with $t=1/2+\mathrm{i}\beta$ and 
using the binomial theorem, we find
\begin{eqnarray}\label{eq:Wd=1}
W_n^{(d)}&=&
 (1/2-\mathrm{i}\beta)^{u_n}(1/2+\mathrm{i}\beta)^{n-u_n+1}\nonumber\\
&=&\left(\sum_k\binom{u_n}{k}(-1)^k(\mathrm{i}\beta)^k
    \frac{1}{2^{u_n-k}}\right) \times
    \left(\sum_\ell\binom{n-u_n+1}{\ell}
    (\mathrm{i}\beta)^\ell\frac{1}{2^{n-u_n+1-\ell}}\right)\nonumber\\
&=&\sum_k(-1)^{k}\binom{u_n}{k}\sum_\ell\binom{n-u_n+1}{\ell}
  \frac{1}{2^{n+1-k-\ell}}(\mathrm{i}\beta)^{k+\ell}.
\end{eqnarray}
Letting $k+\ell=m$ in (\ref{eq:Wd=1}) we obtain
\begin{equation*}\label{eq:w}
w_{m,n}=\frac{\mathrm{i}^m}{2^{n+1-m}}\sum_k (-1)^k\binom{u_n}{k}\binom{n+1-u_n}{m-k}.
\end{equation*}
For $k<0$ ($k>m$) the first (second) binomial coefficient is zero, 
so we may restrict the range of summation to $0\leqslant k\leqslant m$,
which completes the proof of the recursion formulae.
\end{proof}

\begin{lemma}\label{lma:Convergence}
Let $(a_{k,n}^{(d)})$ be the coefficients of $Z_{n}^{(d)}$, as in 
equation (\ref{eq:Zpoly}). Then for all $d$ and $k$ the following limit exists
\begin{equation}\label{eq:Convergence}
a_{k}^{(d)}:=\lim_{n\to\infty} a_{k,n}^{(d)},\qquad
\mbox{with the bound }\qquad |a_k^{(d)}|<2^{k+1}
\end{equation}
independent on $d$. Thus the polynomials $Z_n^{(d)}$ converge coefficientwise
to a convergent power series:
\begin{equation}\label{eq:Z}
Z^{(d)}(\beta):=\lim_{n\to\infty}Z_n^{(d)}(\beta),
\qquad |\beta|<\frac{1}{2}.
\end{equation}
\end{lemma}
\begin{proof}
Using the identities \citep[pp 174,199]{GKP}
$$
\sum_k\binom{r}{k}\binom{s}{m-k}=\binom{r+s}{m},
\hskip 30pt
\frac{z^m}{(1-z)^{m+1}}=\sum_{k\geqslant 0}\binom{k}{m}z^k 
\quad (m\geqslant 0), 
$$
and formula (\ref{eq:W}) we obtain the following uniform bound
(to lighten up the notation, we write $a_{k,n}$ for $a_{k,n}^{(d)}$):
\begin{eqnarray}
|a_{m,n+1}|&=&|a_{m,n+1}-a_{m,0}|=\left|\sum_{j=0}^n (a_{m,j+1}-a_{m,j} )\right|
\nonumber\\
&\leqslant&\sum_{j=0}^n\bigl|a_{m,j+1}-a_{m,j}\bigr|=\sum_{j=0}^n|w_{m,j}|
\nonumber\\
&=& \sum_{j=0}^n \frac{2^m}{2^{j+1}}
    \sum_k\binom{u_j}{k}\binom{j+1-u_j}{m-k}
\nonumber\\
&=& 2^m\sum_{j=0}^n \frac{1}{2^{j+1}}\binom{j+1}{m}
\nonumber\\
&<& 2^m\sum_{j=0}^\infty \frac{1}{2^{j+1}}\binom{j+1}{m}
=2^{m+1}.\nonumber
\end{eqnarray}
This derivation shows that, for all $d\in\Omega$, the series 
$\sum_{j}w_{m,j}$ converges absolutely, which implies 
that, as $n\to\infty$, the sequence of polynomials $Z_n^{(d)}$ 
converges coefficientwise to a power series; the latter in turn 
converges for $|\beta|<1/2$ for all $d$.
\end{proof}

We shall use the following lemma (see \citep[p.~19--20]{Allaart}).
\begin{lemma}\label{lma:Tak1}
For any sequence $d\in\Omega$ let $r_n^{(d)}$ and $u_n^{(d)}$ be as 
in (\ref{eq:ru}), and let $\rT$ be the Takagi function.
Then for any $m\geqslant n$:
\begin{equation}\label{eq:lemTak1}
    \rT\big(r_n^{(d)}+\frac{1}{2^m}\big)-\rT\big(r_n^{(d)}\big)
           =\frac{m-2u_n^{(d)}}{2^m}.
\end{equation}
\end{lemma}

\begin{figure}[ht]
\begin{center}
\includegraphics[width=0.7\textwidth]{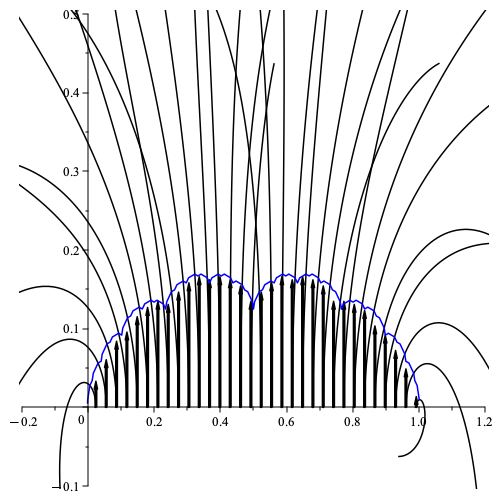}
\caption{The parametric curves $Z_n^{(d)}(\beta)$ of (\ref{eq:Z_n}) 
in the complex plane, for all binary sequences of length $n=2^6$, 
and $\beta\in[0,1/4]$. The arrow tangent to each curve at 
$\alpha=Z_n^{(d)}(0)$ represents the vector field $\bv(\alpha)$ scaled 
by the value $\beta=1/8$, which is a linear approximation
of the attractor $\mathcal{A}(1/8)$ (the transversal curve).
Intersections of parametric curves correspond to the appearance 
of points with multiple addresses in the attractor.
}\label{fig:VectorField}
\end{center}
\end{figure}

The connection between B\'ezier curves and the Takagi function first appears
in the following result.

\begin{lemma}\label{lma:v}
Let $Z^{(d)}(\beta)$ be as in (\ref{eq:Z}). Then, for every $\alpha\in[0,1]$:
\begin{equation}\label{eq:v}
\mathbf{v}(\alpha):=\left.
    \frac{\mathrm{d}Z^{(\pi^{-1}(\alpha))}(\beta)}{\mathrm{d}\beta}
            \right\vert_{0}=2\ri\rT(\alpha),
\end{equation}
where $\pi$ is defined in (\ref{eq:pi}) and $\rT$ is the Takagi function.
\end{lemma}

\begin{proof}
For any $n\geqslant 1$, from equation (\ref{eq:Zpoly}) we have
$$
\left.\frac{\mathrm{d}Z_n^{(d)}(\beta)}{\mathrm{d}\beta}
            \right\vert_{0}=\ri a_{1,n}^{(d)},
$$
so we have to compute the limits $a_0^{(d)}$ and $a_1^{(d)}$ as in 
(\ref{eq:Convergence}).
We determine the first coefficient $a_{0,n}$ of the 
polynomial $Z_n^{(d)}(\beta)$.

If $\beta=0$, the IFS (\ref{eq:f}) is represented as the pair
of planar maps
\begin{equation}\label{eq;F0}
F_0(x,y)=\bigl(\frac{1}{2}x\,,\, \frac{1}{2} y\bigr)
\hskip 30pt
F_1(x,y)=\bigl(\frac{1}{2}x+\frac{1}{2}\,,\, \frac{1}{2} y\bigr).
\end{equation}
The attractor $\mathcal{A}(0)$ is the unit interval, over which the
map is seen to be the inverse of the doubling map $x\mapsto 2x\mod{1}$.
Let $d(n)=(d_1,d_2,\ldots,d_n)$ be the sequence of
the first $n$ digits of $d$, and
consider expression (\ref{eq:f_n}) for the finite 
sequence $\revd(n)$. Since $\beta=0$, the final point 
will be $Z_n^{(d)}(0)=a_{0,n}$, where $a_{0,n}$ 
is a rational number. 

Consider now the orbit of the point $z=0$ under the sequence of maps
$f^{(k)}$ determined according to the expression (\ref{eq:f^d}).
This sequence is $\revd(n)$, and the final point is 
$Z_n^{(d)}(0)=a_{n,0}^{(d)}$, from (\ref{eq:f_n}).
This point is the initial point of an orbit of the inverse map
---the doubling map-- which corresponds to the reverse sequence $d(n)$.
It is well-known that the sequence $d(n)$ is the sequence of binary 
digits of the initial point, and therefore we have
\begin{equation}\label{eq:a_0}
a_{0,n}=r_n^{(d)}=\sum_{k=1}^{n}\frac{d_{k}}{2^k}\, \in\, \Z[1/2].
\end{equation}
In other words, the dyadic rational $r_n^{(d)}$ is the initial point of an 
orbit of the doubling map, which terminates at the fixed point $0$ 
in $n$ steps. 
Taking the limit as $n\to\infty$, from (\ref{eq:limit}) we find 
that $a_0^{(d)}=x^{(d)}$.

Next we determine $a_{1,n}$, using induction on $n$. For $n=1$ we have
$f^{(0)}(0)=0$ and $f^{(1)}(0)=1/2+\mathrm{i}\beta = a_{0,1} + \mathrm{i}\beta a_{1,1}$, 
so in either case we have $a_{1,1}=2\rT(a_{0,1})$. Assume now that this
is true for all binary sequences of length $n$: $a_{1,n}=2\rT(a_{0,n})$. 

We consider the successive digit $d_{n+1}$. If $d_{n+1}=0$ we have
$a_{1,n+1}=a_{1,n}$ from (\ref{eq:d=0}), and there is nothing to prove.
If $d_{n+1}=1$, then from (\ref{eq:Zrecursion}) and (\ref{eq:W}) with
$m=1$ we find
\begin{equation}\label{eq:bn+1}
a_{1,n+1}=-\mathrm{i}w_{1,n}+a_{1,n}.
\end{equation}
We compute
\begin{equation*}
-\mathrm{i}w_{1,n}=\frac{1}{2^n} [(n+1-u_n)-u_n]=\frac{n+1-2u_n}{2^n}.
\end{equation*}
From the above, equation (\ref{eq:bn+1}), the induction hypothesis, 
and lemma \ref{lma:Tak1} with $m=n+1$, we have
\begin{equation*}
a_{1,n+1}=a_{1,n}+\frac{n+1-2u_n}{2^n}
 = 2\rT(a_{0,n})+\frac{n+1-2u_n}{2^n}
 = 2\rT\bigl(a_{0,n}+\frac{1}{2^{n+1}}\bigr),
\end{equation*}
which completes the induction.
Taking the limit, we have, for any $d$
$$
\lim_{n\to\infty} a_{1,n+1}^{(d)}=\lim_{n\to\infty}2\rT
  \bigl(a_{0,n}^{(d)}+\frac{d_{n+1}}{2^{n+1}}\bigr)
=2\rT(\lim_{n\to\infty} a_{0,n+1}^{(d)})=2\rT(x^{(d)}),
$$
the penultimate step being justified by the continuity of $\rT$.
The proof is complete.
\end{proof}

From lemmas \ref{lma:Convergence} and \ref{lma:v}
we obtain an infinite sequence of functions
\begin{equation}\label{eq:a_k}
a_k:[0,1]\to \R \qquad \alpha\mapsto a_k^{(\pi^{-1}(\alpha))}
\qquad k=0,1,\ldots,
\end{equation}
with $a_0(\alpha)=\alpha$, $a_1(\alpha)=2\rT(\alpha)$, and $a_k$ displaying an
increasing degree of irregularity as $k$ increases (see figure
\ref{fig:Coefficients}).

\begin{figure}[ht]
\begin{center}
\includegraphics[width=0.6\textwidth]{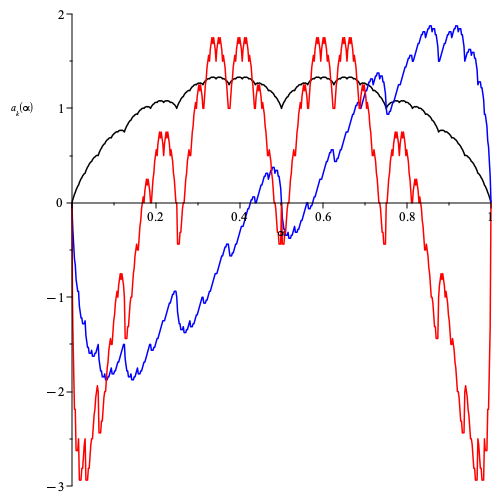}
\caption{The graphs of the functions $\alpha\mapsto a_k(\alpha)$ of (\ref{eq:a_k}), 
for $k=1$ (black), $k=2$ (blue) and $k=3$ (red).
The function $a_1$ is the scaled Takagi function ---cf.~lemma \ref{lma:v}.
}\label{fig:Coefficients}
\end{center}
\end{figure}

The quantity $\bv$ is the simplest instance of a non-smooth vector 
field on a smooth B\'ezier curve $\cB=\mathcal{A}(0)$ 
(see figure \ref{fig:VectorField}). 
The field $\bv$ is orthogonal to the curve, and arises from the 
complexification of the curve's parameter: $t=1/2+\ri\beta$. 
In this context the function $\rg$ performs scaling in the
direction of the vector field.

The set 
$$
\bigcup_\alpha \bigl(\alpha + \bv(\alpha)\beta\bigr),
$$
represents a linear approximation of $\cA(\beta)$ 
(see figure \ref{fig:VectorField}), which improves as $\beta\to 0$. 
To express this fact more precisely, we scale $\cA$ in 
the direction of the field. This result is the content of the following
theorem.

\begin{theorem}\label{thm:Takagi}
Let $\mathcal{A}(\beta) \in \mathbb{C}$ be the attractor of the 
de Casteljau IFS (\ref{eq:f}) for the B\'ezier curve with control 
points $0$ and $1$ and complex subdivision parameter $t=\frac{1}{2}+i\beta$,
and let $\rg$ be the scaling map (\ref{eq:g}).
Then 
\begin{equation}\label{eq:TakagiTheorem}
\lim_{\beta\to 0} \rd_H(\rg\big(\mathcal{A}(\beta),(2\beta)^{-1}\big),\cT)=0,
\end{equation}
where $\cT$ is the graph of the Takagi function and $\rd_H$ is the
Hausdorff distance (\ref{eq:HausDistance}).
\end{theorem}

\begin{proof}
We shall provide two proofs.
\smallskip

From lemma \ref{lma:Convergence}, we can represent the attractor of
the IFS as
$$
\cA(\beta)=\bigcup_{d\in\Omega} Z^{(d)}(\beta)
=\bigcup_{\alpha\in[0,1]} Z^{(\pi^{-1}(\alpha))}(\beta)\qquad |\beta|<1/2.
$$
Let $\cA^*(\beta)=\rg(\cA(\beta),(2\beta)^{-1})$, and 
write $Z(\alpha,\beta)$ for $Z^{(\pi^{-1}(\alpha)}(\beta)$.
We have
\begin{eqnarray*}
\cA^*(\beta)&=&\bigcup_\alpha \rg(Z(\alpha,\beta),(2\beta)^{-1})
=\bigcup_\alpha \alpha + \rg(Z(\alpha,\beta)-\alpha,(2\beta)^{-1})\\
&=&\bigcup_\alpha \alpha + \frac{1}{2}\bv(\alpha) + O_\alpha(\beta)\\
&=&\cT + \bigcup_{\alpha}O_\alpha(\beta),
\end{eqnarray*}
where $|O_\alpha(\beta)|\leqslant C\beta$, for some constant
$C=C(\alpha)=C_x(\alpha)+\ri C_y(\alpha)$. (Here `$+$' represents
the Minkowski sum of sets.)
We have
$$
\rd_H(\cA^*(\beta),\cT)\leqslant \sup_{\alpha}C(\alpha)\beta.
$$
Using \ref{eq:Zpoly} and (\ref{eq:Convergence}) we estimate:
\begin{eqnarray*}
|C(\alpha)|&\leqslant&|C_x(\alpha)|+|C_y(\alpha)|\leqslant\sum_{k\geqslant 1}
    \bigl(|a_{2k}|+\frac{1}{2}|a_{2k+1}|\bigr)\beta^{2k-2}\\
&\leqslant&\sum_{k\geqslant 1}(2^{2k+1}+2^{2k+1})\beta^{2k-2}\\
&=&4\sum_{k\geqslant 0}(4\beta^2)^k=\frac{4}{1-4\beta^2},
    \qquad |\beta|<\frac{1}{2}.
\end{eqnarray*}
Thus, for $|\beta|<1/2$ we have
$$
\rd_H(\cA^*(\beta),\cT)\leqslant\frac{4\beta}{1-4\beta^2}\to 0,
$$
as claimed.
\bigskip

For an alternative proof, we let $\beta\in(0,\sqrt{3}/2)$, in which range
the IFS $\{f^{(0)},f^{(1)}\}$ is hyperbolic 
[see comment after eq.~(\ref{eq:f})]. 
For any $d\in\Omega$, we define the 
following sequence of complex numbers:
\begin{equation}\label{eq:frev}
\revZ_0^{(d)}(\beta)=0\qquad
\revZ_n^{(d)}(\beta)=
  f^{(d_n)}\circ f^{(d_{n-1})}\circ\cdots\circ f^{(d_1)}(0),\quad n\geqslant 1.
\end{equation}
[Note that in general $\revZ_n^{(d)}\not=Z_n^{(\revd)}$, cf.~(\ref{eq:revd}).]
Since $0$ belongs to $\mathcal{A}(\beta)$ for any $\beta$, so do all the 
$\revZ_n^{(d)}$ and $Z_n^{(d)}$. It follows that the two sets 
\begin{equation*}\label{eq:Gamma}
\Gamma_n(\beta)=\bigcup_{d\in\Omega}Z_n^{(d)}(\beta)
\hskip 30pt
\revGamma_n^{(d)}(\beta)=\bigcup_{k=0}^n \revZ_k^{(d)}(\beta)
\end{equation*}
are finite subsets of $\mathcal{A}(\beta)$.
Since all binary sequences of length $n$ are represented in $\Gamma_n$,
for any $d$, $n$, and $\beta$ in the range specified above we obtain the 
chain of inclusions
\begin{equation}\label{eq:Inclusions}
\revGamma_n^{(d)}(\beta)\subset\Gamma_n(\beta)\subset \mathcal{A}(\beta).
\end{equation}
From the hyperbolicity of the IFS and 
\citep[theorem 2, p.~365]{FractalsEverywhere} it then follows
that for almost all $d\in\Omega$ [with respect to the Lebesgue
measure, see comment following eq.~(\ref{eq:pi})]) we have
\begin{equation*}
\overline{\lim_{n\to\infty} \revGamma_n^{(d)}(\beta)}=\mathcal{A}(\beta),
\end{equation*}
where the overbar denotes the closure.
From (\ref{eq:Inclusions}) we then have
\begin{equation}\label{eq:Closure}
\overline{\lim_{n\to\infty} \Gamma_n(\beta)}=\mathcal{A}(\beta)
  \qquad  \beta\in (0,\sqrt{3}/2).
\end{equation}

From lemma \ref{lma:v}, we have 
\begin{equation}\label{eq:O}
\rg(Z_n^{(d)}(\beta),(2\beta)^{-1})
    =r_n^{(d)}+\mathrm{i}\rT(r_n^{(d)})+O(\beta^2),\qquad \beta\to 0,
\end{equation}
where the constant in $O(\beta^2)$ depends on $d$ and $n$.
We obtain
$$
\lim_{\beta\to 0} \mathrm{d}(\rg \big(\Gamma_n(\beta),(2\beta)^{-1}\big),\cT)=0,
$$
where the (asymmetric) distance $\mathrm{d}$ is defined in (\ref{eq:dist2}).
Taking the limit as $n\to\infty$ gives a set of points of $\mathcal{A}^*$
which are dense in $\cT$. Since $\mathcal{A}^*$ is closed, their closure $\cT$
is contained in $\mathcal{A}^*$.
Our result now follows from the chain of inclusions (\ref{eq:Inclusions}).
\end{proof}

\begin{remark}
Theorem \ref{thm:Takagi} does not imply that $\mathcal{A}^*$ 
(which is connected from theorem \ref{thm:IFSAttractor})
is topologically equivalent to $\cT$ for all sufficiently small $\beta$.
Indeed while there is a unique parametric curve $Z^{(d)}(\beta)$ through each point 
of $\cB=\cA(0)$, the union of all curves does not constitute a fiber bundle 
over $\cB$ because the intersections of curves visible in figure 
\ref{fig:VectorField} can be shown to occur arbitrarily close to $\cB$,
for instance, near $x=1/2$.
\end{remark}

\medskip
In closing, we show (informally) that the Takagi function also appears 
in the general case of a B\'ezier curve of degree $m\geqslant 2$, 
as one component of the vector field. 
We use the lower-triangular representation (\ref{eq:ltM}) 
of the de Casteljau matrices.
Let $\{f^{(0)},f^{(1)}\}$ be the corresponding
affine IFS on $\C^m$, and let $\mathcal{\bar A}(\beta)$ 
be the attractor for $t=1/2+\mathrm{i}\beta$.
For any sequence $d\in\Omega$ we form the matrix product 
$\mathbf{M}_n(d)$ as in (\ref{eq:M_n}).
Then the vector
\begin{equation}\label{eq:Zvector}
 \left(\begin{array}{c} 1\\
                        Z_{n,1}^{(d)}\\
                          \vdots\\
                         Z_{n,m}^{(d)}
\end{array}\right)
=\mathbf{M}
 \left(\begin{array}{c} 1\\
                        0\\
                        \vdots\\
                        0
\end{array}\right),
\end{equation}
whose entries are polynomials in $\Z[1/2,\ri\beta]$, belongs to 
$\mathcal{\bar A}(\beta)$ for any $\beta$, since so does the 
vector $(1,0^m)^\top$, being the fixed point 
of $f_0$ in homogeneous coordinates. 

Consider now the first component of this vector.
From the lower triangular form of $\mathbf{M}^{(0)}$ and $\mathbf{M}^{(1)}$ 
it is straighforward to show that its evolution is determined by the 
pair of maps 
\begin{equation}\label{eq:f_m}
f^{(0,m)}(x)=tx\qquad 
 f^{(1,m)}(x)=(1-t)x+mt\hskip 30pt t=\frac{1}{2}+\mathrm{i}\beta.
\end{equation}
This system is conjugate to the system (\ref{eq:f}) by the bijection
$x\mapsto mx$. 
Hence for $\beta=0$ the smooth B\'ezier curve $\cB_{1}(0)$ is the interval 
with end-points 0 and $m$, and the map $\pi_{m}$ which sends the
symbolic binary address $d$ to the corresponding point on 
$\cB_1$ (namely the constant coefficient of $Z_{n,1}^{(d)}$), is given 
by $\pi_{m}=m\pi$, where $\pi$ is defined in (\ref{eq:pi}).
In the limit as $n\to\infty$ the first (affine) component of the vector 
(\ref{eq:Zvector}) is equal to $mZ^{(d)}$, and hence the 
corresponding component of the vector field is given by
\begin{equation}\label{eq:v_m}
x=\pi_{m}(d)=m\sum_{k\geqslant 1}\frac{d_k}{2^k},
\hskip 30pt
\bv_{m}(x)=\left.
    \frac{\mathrm{d}Z^{(\pi_{m}^{-1}(x))}(\beta)}{\mathrm{d}\beta}
            \right\vert_{0} =2m\ri\rT(x/m),
\end{equation}
where $\rT$ is the Takagi function. 
Defining the scaling function as
$$
\rg(z,\alpha,\beta)=\alpha\Re(z)+ \beta\,\ri\,\Im(z)
$$
the limit (\ref{eq:TakagiTheorem}) becomes
\begin{equation*}
\lim_{\beta\to 0} 
  \rd_H(\rg\big(\mathcal{A}(\beta),1/m,(2m\beta)^{-1}\big),\cT)=0,
\end{equation*}
which is analogous to the case of two control points.


\section*{Declaration of competing interest}
The authors declare that they have no competing interests.

\section*{Acknowledgement}
L. Pt\'a\v{c}kov\'a was supported by the Czech Science Foundation Grant P403-22-11117S.
The detailed comments of a referee helped improve the clarity and accuracy 
of this paper.


\end{document}